\PassOptionsToPackage{dvipsnames, table}{xcolor}
\documentclass[12pt,reqno]{amsart}
\usepackage{amsmath,amssymb,graphicx,mathrsfs,amsopn,stmaryrd,amsthm,color,bm,extarrows,ytableau}
\usepackage{newtxtext}
\usepackage{float}
\usepackage[dvipsnames]{xcolor}
\usepackage[colorlinks=true,citecolor=blue,linkcolor=blue]{hyperref}
\usepackage[alphabetic,nobysame]{amsrefs}
\renewcommand{\MR}[1]{} 
\usepackage[english]{babel}
\usepackage[mathscr]{euscript}
\usepackage[left=2.5cm,
            right=2.5cm,
            bottom=2.75cm
            ]{geometry}
\DeclareFontFamily{OT1}{pzc}{}
\DeclareFontShape{OT1}{pzc}{m}{it}{ <-> s*[1.0] pzcmi7t }{}
\DeclareMathAlphabet{\mathpzc}{OT1}{pzc}{m}{it}

\usepackage{collectbox}


\usepackage{dsfont}
\usepackage{tikz}    
\usetikzlibrary{arrows,matrix,decorations.markings,shapes.geometric}   
\usetikzlibrary{decorations.pathreplacing}

\usepackage[all]{xy}
\allowdisplaybreaks
\xyoption{all}

\numberwithin{equation}{section}
\newtheorem{thm}{Theorem}[section]
\newtheorem{prop}[thm]{Proposition}
\newtheorem{lem}[thm]{Lemma}
\newtheorem{cor}[thm]{Corollary}
\newtheorem{conj}[thm]{Conjecture}

\theoremstyle{definition} 

\newtheorem{dfn}[thm]{Definition}

\newtheorem{alphatheorem}{Theorem}

\newtheorem{alphaproposition}[alphatheorem]{Proposition}

\theoremstyle{remark}
\newtheorem{rem}[thm]{Remark}

\newcommand{\beq}{\begin{equation}}
\newcommand{\eeq}{\end{equation}}
\newcommand{\be}{\begin{equation*}}
\newcommand{\ee}{\end{equation*}}

\newcommand{\bC}{\mathbb{C}}

\newcommand{\bZ}{\mathbb{Z}}

\newcommand{\bN}{\mathbb{N}}

\newcommand{\mc}{\mathcal}

\newcommand{\cY}{\mathcal{Y}}

\newcommand{\sfZ}{\mathsf{Z}}

\newcommand{\hf}{\tfrac12}

\newcommand{\cb}{\mathscr{B}}
\newcommand{\ch}{\mathscr{H}}
\newcommand{\cy}{\mathscr{Y}}

\newcommand{\g}{\mathfrak{g}}
\newcommand{\gl}{\mathfrak{gl}}

\newcommand{\fksl}{\mathfrak{sl}}

\newcommand{\fks}{\mathfrak{s}}
\newcommand{\fkv}{\mathfrak{v}}
\newcommand{\fkw}{\mathfrak{w}}

\newcommand{\id}{{\mathrm{id}}}   
   
\newcommand{\gr}{{\mathrm{gr}}}

\newcommand{\tl}{\tilde}
\newcommand{\wtl}{\widetilde}
\newcommand{\gge}{\geqslant}
\newcommand{\lle}{\leqslant}
\newcommand{\la}{\lambda}

\newcommand{\C}{{\mathbb C}}
\newcommand{\W}{{\mathbf W}}
\newcommand{\I}{{\mathbb I}}
\newcommand{\iI}{{}^\imath\I}
\newcommand{\Z}{{\mathbf Z}}
\newcommand{\ve}{\varepsilon}

\newcommand{\Y}{{\mathcal{Y}}}
\newcommand{\Yt}{{^{\imath}\mathcal{Y}}}
\newcommand{\wY}{{}^\imath{\widetilde{\mathcal{Y}}}}

\newcommand{\Tt}{T^\imath}
\newcommand{\Qt}{Q^\imath}
\newcommand{\dfo}{\eth}
\newcommand{\bv}{\mathbf{v}}
\newcommand{\bw}{\mathbf{w}}

\newcommand{\arxiv}[1]{\href{http://arxiv.org/abs/#1}{\tt arXiv:\nolinkurl{#1}}}

\newcommand{\brown}[1]{{\color{Brown}#1}}

\begin{document}
\pagestyle{myheadings}
\setcounter{page}{1}

\title[Shifted twisted Yangians]{Shifted twisted Yangians of quasi-split ADE types}

\author{Kang Lu}
\author{Weiqiang Wang}
\address{Department of Mathematics, University of Virginia, 
Charlottesville, VA 22903, USA}\email{kang.lu.math@gmail.com, ww9c@virginia.edu}
\author{Alex Weekes}
\address{Département de mathématiques, Université de  Sherbrooke, Sherbrooke, QC, Canada}
\email{alex.weekes@usherbrooke.ca}

\subjclass[2020]{Primary 17B37.}
\keywords{Drinfeld presentation, twisted Yangians}

\begin{abstract}
Associated to all quasi-split Satake diagrams of type ADE and even spherical coweights $\mu$, we introduce the shifted iYangians ${}^\imath \mathcal Y_\mu$ and establish their PBW bases. We construct the iGKLO representations of ${}^\imath \mathcal Y_\mu$, which factor through quotients called truncated shifted iYangians ${}^\imath \mathcal Y_\mu^\lambda$. In type AI with $\mu$ dominant, a variant of ${}^\imath \mathcal Y_\mu^{N\varpi_1^\vee}$ is identified with the truncated shifted iYangians  in another definition, which are isomorphic to finite W-algebras of type BCD. These new family of algebras has connections and applications to fixed point loci of affine Grassmannian slices which will be developed in a sequel. 
\end{abstract}
	
\maketitle

\setcounter{tocdepth}{1}
\tableofcontents

\thispagestyle{empty}

\section{Introduction}

\subsection{The goal}
Yangians $Y(\g)$, for complex simple Lie algebras $\g$ or $\g=\gl_N$,  admit natural variants known as shifted Yangians $Y_\mu$ associated to coweights $\mu$, which are actually subalgebras of $Y(\g)$ for $\mu$ dominant. Shifted Yangians in the dominant cases were first introduced in type A by Brundan-Kleshchev \cite{BK06}, who showed certain quotient algebra $Y_\mu^{N\varpi_1^\vee}(\gl_N)$ is isomorphic with a finite W-algebra of type A. Inspired by \cite{GKLO}, a general family of truncated shifted Yangians $Y_\mu^\la$ for general $\g$ of type ADE were introduced in \cite{KWWY14} for $\mu$ dominant, and for general $\mu$ in \cite{BFN19}, and they were shown to quantize generalized affine Grassmannian slices \cite{BFN19,W19}.  The aforementioned BK identification can then be viewed as an instance of quantizations of Poisson algebra isomorphisms between certain affine Grassmannian slices and nilpotent Slodowy slices \cite{WWY20,MV22}. 

In this paper and its sequel, we shall add a new twist to the above connections and the $\imath$Program (see \cite{WangICM}), providing a  framework in which new families of algebras quantize new Poisson varieties. This paper introduces shifted iYangians and studies their iGKLO representations, laying an algebraic foundation for the geometric applications to affine Grassmannian islices in the sequel \cite{LWW25islice}.

\subsection{Shifted iYangians}

Associated with Satake diagrams, Letzter formulated quantum symmetric pairs $(\bf U,\bf U^\imath)$ and a generalization to Kac-Moody type was given by Kolb, where $\bf U$ is a Drinfeld-Jimbo quantum group and $\bf U^\imath$ is a coideal subalgebra of $\bf U$ known as an iquantum group nowadays. The quantum group $\bf U$ appears as an iquantum group associated to the diagonal Satake diagram. The $\imath$-generalization of affine quantum groups and Yangians are known as affine iquantum groups and twisted Yangians (or iYangians). Twisted Yangians were defined earlier in \cite{Ols92, MR02, GuayR16} in R-matrix form, deforming the twisted current algebras (i.e., involutive fixed point subalgebras of current algebras).

Recently, a Drinfeld presentation of the twisted Yangian of type AI has been constructed in \cite{LWZ25} via Gauss decomposition; such presentations have been extended to twisted Yangians of all split types \cite{LWZ25degen} and then to all quasi-split types \cite{LZ24} via a degeneration construction from the corresponding Drinfeld presentation of (quasi-)split affine iquantum groups \cite{LW21, mLWZ24}. 
However, the identification of the two definitions of twisted Yangians (one via R-matrix, and the other via Drinfeld presentation) is only established in type AI and AIII, and in this paper we shall follow the ones in Drinfeld presentation and refer to them as iYangians. These iYangians and affine iquantum groups are determined by the underlying finite type quasi-split Satake diagrams $(\I,\tau)$. Here, $\tau$ is a diagram involution of the Dynkin diagram $\I$; the cases with $\tau=\id$ are called split. 

Associated to an arbitrary quasi-split Satake diagram $(\I,\tau)$, we have a symmetric pair $(\g, \g^{\omega_\tau})$, where $\omega_\tau =\omega_0\circ\tau$ for the Chevalley involution $\omega_0$, and the corresponding iYangian $\Yt=\Yt_0$. Modifying the iYangians in Drinfeld presentation, we introduce the shifted iYangians $\Yt_\mu$; see Definition \ref{def:qsplit}. Here $\mu$ is an arbitrary even spherical weight (see Definition \ref{def:spherical}), some natural conditions imposed by necessity of the iGKLO representations and the geometric properties of affine Grassmannian islices. Note that a  family of shifted iYangians  $\Yt_\mu$ of type AI, for $\mu$ dominant and even (automatically spherical since $\tau=\id$), has  appeared and played a fundamental role in \cite{LPTTW25}.

\begin{alphatheorem}
[PBW basis Theorem \ref{thm:pbw_arb}] \label{thm:pbw:Intr}
Let $\mu$ be an arbitrary even spherical coweight. 
Then the set of ordered monomials in the root vectors \eqref{rvset3} forms a basis for $\Yt_\mu$. 
\end{alphatheorem}

The shift homomorphisms between shifted iYangians, $\iota^\tau_{\mu,\nu}:\Yt_\mu\to \Yt_{\mu+\nu+\tau \nu}$ (see Lemma~ \ref{lem:shiftqs}), for $\nu$ anti-dominant and $\mu, \nu+\tau \nu$ even spherical, are injective as a consequence of Theorem~\ref{thm:pbw:Intr}.

\subsection{iGKLO and truncated shifted iYangians}

Recall the GKLO homomorphisms from shifted Yangians to a ring of difference operators \cite{GKLO, KWWY14, BFN19}, with images called truncated shifted Yangians. We define a variant of such a ring of difference operators $\mc A$ as in \eqref{A:generators}--\eqref{A:relations}. We construct iGKLO homomorphisms from shifted iYangians to $\mc A$, first for $\Yt_\mu(\gl_n)$ in Theorem \ref{thm:iGKLO:gl}, and then for general $\Yt_\mu$. 

\begin{alphatheorem} [Theorem \ref{thm:GKLOquasisplit}]
    \label{thm:GKLO:Intr}
Let $(\I,\tau)$ be any quasi-split Satake diagram. Let $\la$ be a dominant $\tau$-invariant coweight and $\mu$ be an even spherical coweight subject to the parity constraint \eqref{parity}. 
Then there exists a unique homomorphism 
$\Phi_\mu^\la:\Yt_\mu[\bm z]\longrightarrow \mc A$ with prescribed formulas on generators. 
\end{alphatheorem}

Finding the explicit formulas for the images of generators of $\Yt_\mu$ under $\Phi_\mu^\la$ is part of the main challenges here, as we work in complete generality for all types. Note that our formulas feature some inhomogeneous terms for general $\la, \mu$, as demanded by the geometric application in the sequel \cite{LWW25islice}. 
It is worth noting that the same parity condition \eqref{parity} appears naturally in the geometric setting {loc.~cit.}. 

The proof of Theorem \ref{thm:GKLO:Intr} is long and will occupy Section \ref{sec:proof1}. Since there are 3 affine rank one subalgebras and hence $\Yt_\mu$ admits several different types of Serre relations which require different computations, it is a very tedious and technical challenge to verify that $\Phi_\mu^\la:\Yt_\mu[\bm z]\rightarrow \mc A$ is a homomorphism. 

Associated with the Satake diagram of type AIII$_{2n}$ with an even number of nodes, a variant of the shifted iYangians based on \cite{LZ24} is also formulated in \cite{SSX25} though no PBW basis is established; the authors also construct a homomorphism to the ring $\mc A$, which factors through a corresponding Coulomb branch. In this case, we learned from these authors a generating function trick which allowed us to verify a Serre relation \eqref{SerreIII2} required for the iGKLO homomorphism. Type AIII$_{2n}$ is special among all ADE Satake diagrams in that it contains no split rank one subdiagram.

A truncated shifted iYangian (TSTY) of type $(\I,\tau)$ is by definition the image of the homomorphism $\Phi_\mu^\la$ and will be denoted by $\Yt_\mu^\la$. This family of algebras will play a fundamental role in the sequel, and they admit several favorable properties; here is one of them. 

\begin{alphaproposition}
    [Proposition \ref{prop:center}]
    \label{prop:centerTSTY:Intr}
    The center of the TSTY $\Yt_\mu^\la$ is a polynomial algebra.
\end{alphaproposition}

\subsection{A tale of two TSTY's}

Via a parabolic generalization of the Drinfeld presentation of twisted Yangians of type AI \cite{LWZ25}, a very different definition of shifted iYangians $\cY_\mu$ has been given in \cite{LPTTW25}, where $\mu$ is a partition assumed to be even in the sense that its associated nilpotent element gives rise to an {\em even} $\Z$-grading on the underlying classical Lie algebra. Miraculously, this notion of evenness is compatible with the evenness condition on $\mu$ as a dominant coweight (valid for all ADE type) used in this paper. A family of truncated shifted iYangians (TSTY) of type AI, denoted by $\cy_{n,\ell}^+(\sigma)$, was defined very differently in \cite{LPTTW25} (generalizing \cite{BK06}) in terms of generators and relations, and one main result {\em loc. cit.} is an algebra isomorphism between $\cy_{n,\ell}^+(\sigma)$ and finite W-algebras, quantizing the corresponding Slodowy slices of type BCD. 

It is a simple matter to align the combinatorial data $(n,\ell;\sigma)$ used in $\cy_{n,\ell}^+(\sigma)$ and a pair $(N\varpi_1^\vee, \mu)$ used in TSTY $\Yt_\mu^{N\varpi_1^\vee}$ here. 
We introduce a variant of  $\Yt_\mu^{N\varpi_1^\vee}$, denoted by $\wY_\mu^{N\varpi_1^\vee}$ in \eqref{def:tildeTSTY}.

\begin{alphatheorem} [Theorem \ref{thm:TruncatedSTY}]
    \label{thm:TSTY:Intr}
    These two versions of TSTY's are isomorphic: $\cy_{n,\ell}^+(\sigma) \cong \wY_\mu^{N\varpi_1^\vee}$.
\end{alphatheorem}
This in particular gives a presentation for the TSTY $\wY_\mu^{N\varpi_1^\vee}$. It remains an open problem to give a presentations for the TSTY $\Yt_\mu^\la$ in general; see Conjecture \ref{conj:presentation}. 

As a $q$-analogue of shifted Yangians, shifted affine quantum groups of type A are constructed in \cite{FT19} and mapped homomorphically into the quantized K-theoretic Coulomb branches of framed quiver gauge theories. In an ongoing work, we have been able to formulate the shifted affine iquantum groups of split and quasi-split types and their truncations via iGKLO representations.


\subsection{Organization}
The paper is organized as follows.
In Section \ref{sec:S.T.Yangians}, we formulate the shifted iYangians and establish their PBW bases. 
We construct the iGKLO representations of shifted iYangians in Section \ref{sec:GKLO}, which allow us to define the TSTY $\Yt_\mu^\la$. Then we determine the center of $\Yt_\mu^\la$. 
Section \ref{sec:proof1} provides the long technical proof of Theorem \ref{thm:GKLO:Intr}. Finally, in Section \ref{sec:2TSTY}, we identify the TSTY of type AI formulated in \cite{LPTTW25} with a variant of the TSTY defined through iGKLO.

\vspace{2mm}
\noindent {\bf Acknowledgement.}
K.L. and W.W. are partially supported by the NSF grant DMS-2401351. A.W.~is supported by an NSERC Discovery Grant.  W.W. thanks ICMS at Edinburgh for providing a stimulating atmosphere in August 2023 when some of the connections was conceived, and National University of Singapore (Department of Mathematics and IMS) for providing an excellent research environment and support at the final stage of this work. The main results of this work and its sequel were announced in \cite{LPTTW25}. We thank Yaolong Shen, Changjian Su and Rui Xiong for communicating their work with us; we thank them and Curtis Wendlandt for helpful discussions.



\section{Shifted iYangians}
\label{sec:S.T.Yangians}

In this section, we introduce shifted iYangians $\Yt_\mu$ associated to all quasi-split ADE Satake diagrams $(\I,\tau)$ and even spherical coweights $\mu$. We further establish their PBW bases and shift morphisms between shifted iYangians.

\subsection{Shifted iYangians of quasi-split ADE type}

Let $C=(c_{ij})_{i,j\in \I}$ be the Cartan matrix of type ADE, and let $\g$ be the corresponding simple Lie algebra. We fix a simple system $\{\alpha_i\mid i\in \I\}$ with corresponding set $\Delta^+$ of positive roots. Let $\tau$ be an involution of the  Dynkin diagram of $\g$, i.e., $c_{ij}=c_{\tau i,\tau j}$ such that $\tau^2=\mathrm{id}$; note that $\tau=\id$ is allowed. We refer to $(\I,\tau)$ as quasi-split Satake diagrams and call the Satake diagrams split if $\tau=\id$. The split Satake diagrams formally look identical to Dynkin diagrams.
Denote $\I_0$ the set of fixed points of $\tau$ in $\I$, i.e., $\I_0 =\{i\in \I \mid \tau i=i\}$. Let $\I_1$ be a set of representatives for $\tau$-orbits in $\I$ of length 2 and define $\I_{-1}=\tau \I_1$; $\I_1$ can be conveniently chosen so that it is underlying Dynkin subdiagram is connected. Then $\I =\I_1 \sqcup \I_0 \sqcup \I_{-1}$. Set \begin{align}  \label{eq:iI}
\iI=\I_1 \sqcup \I_0.
\end{align}

The involution $\tau$ naturally acts on the (co)root and (co)weight lattices of $\g$. A weight or coweight in this paper is always meant to be integral.

\begin{dfn} \label{def:spherical}
A coweight $\mu$ is called {\em spherical} if $\mu =\mu_1 +\tau \mu_1$ for some coweight $\mu_1$. A coweight $\mu$ is \emph{even} if $\mu = 2\mu'$ for some coweight $\mu'$, i.e., $\langle \mu, \alpha_i\rangle \in 2\bZ$ for all $i\in \I$. 
\end{dfn}

We shall assume that a shift coweight $\mu$ to be even spherical in the context of shifted iYangians below, and remarkably, the same condition will be needed in consideration of fixed point loci of affine Grassmannian slices later on. 

Denote $[A,B]_+ =AB+BA$. 

\begin{dfn}
\label{def:qsplit}
Let $\mu$ be an even spherical coweight. The \textit{shifted iYangians $\Yt_{\mu}:=\Yt_\mu(\g)$ of quasi-split type} is the $\bC$-algebra with generators $H_{i}^{(r)}$, $ B_{i}^{(s)}$, for $i\in \I$, $r\in\bZ$, and $s\in\bZ_{>0}$, subject to the following relations, for $r,r_1,r_2\in \bZ$ and $s,s_1,s_2 \in \bZ_{>0}$: 
\begin{align}
& H_i^{(r)}=0~\text{ for }~r<-\langle \mu,\alpha_i\rangle,\quad H_{i}^{(-\langle \mu,\alpha_i\rangle)}=1,\label{def0}\\
&[H_i^{(r_1)},H_j^{(r_2)}]=0,\label{hhIII}\\
&[H_i^{(r+2)},B_j^{(s)}]-[H_i^{(r)},B_j^{(s+2)}]=\frac{c_{ij}-c_{\tau i,j}}{2} [H_{i}^{(r+1)},B_j^{(s)}]_+\label{hbNqs} \\
& \hskip4.92cm +\frac{c_{ij}+c_{\tau i,j}}{2} [H_{i}^{(r)},B_j^{(s+1)}]_+ +\frac{c_{ij}c_{\tau i,j}}{4}[H_i^{(r)},B_j^{(s)}], \nonumber\\
&[B_i^{(s_1+1)},B_j^{(s_2)}]-[B_i^{(s_1)},B_j^{(s_2+1)}]=\frac{c_{ij}}2 [B_i^{(s_1)},B_j^{(s_2)}]_+ + 2\delta_{\tau i,j}(-1)^{s_1}H_j^{(s_1+s_2)},\label{bbNqs}
\end{align}
and the Serre relations: for $c_{ij}=0$,
\beq\label{bbtau}
[B_i^{(s_1)},B_{j}^{(s_2)}]=(-1)^{s_1-1}\delta_{\tau i,j}H_{j}^{(s_1+s_2-1)},\\
\eeq
and for $c_{ij}=-1$, $i\ne \tau i\ne j$,
\beq\label{eq:Serre-ord}
\mathrm{Sym}_{s_1,s_2}\big[B_{i}^{(s_1)},[B_{i}^{(s_2)},B_{j}^{(s)}]\big]=0,
\eeq
and for $c_{ij}=-1$, $i=\tau i$,
\beq\label{SerreIII}
\begin{split}
\mathrm{Sym}_{s_1,s_2}\big[B_{i}^{(s_1)},[B_{i}^{(s_2)},B_{j}^{(s)}]\big] 
=
(-1)^{s_1-1}[H_{i}^{(s_1+s_2)},B_j^{(s-1)}],
\end{split}
\eeq
and for $c_{i,\tau i}=-1$,
\beq\label{SerreIII2}
\begin{split}
\mathrm{Sym}_{s_1,s_2}\big[B_{i}^{(s_1)},[B_{i}^{(s_2)},B_{\tau i}^{(s)}]\big]
=
4\,\mathrm{Sym}_{s_1,s_2}(-1)^{s_1-1}\sum_{p\gge 0} 3^{-p-1}[B_{i}^{(s_2+p)},H_{\tau i}^{(s_1+s-p-1)}].
\end{split}
\eeq
Here, if $c_{ij}=-1$ and $s=1$, we use the following convention in \eqref{SerreIII}:
\begin{align*}
[H_{i}^{(r)},B_j^{(s-1)}]
=\sum_{p\gge 0}2^{-2p}\big([H_{i}^{(r-2p-2)},B_{j}^{(s+1)}]-
[H_{i}^{(r-2p-2)},B_{j}^{(s)}]_+\big),
\end{align*}
which follows from \eqref{hbN} if $s>1$, cf. \cite[Lem. 4.11]{LWZ25}.
\end{dfn}

\begin{rem}
If $\tau=\mathrm{id}$, then a spherical coweight $\mu$ is automatically even. In the quasi-split case, it is possible to define shifted iYangians without the evenness condition for $\mu$ by modifying the second relation in \eqref{def0} to $H_{i}^{(-\langle \mu,\alpha_i\rangle)}=(-1)^{\langle \mu_1,\alpha_{\tau i}\rangle}$, where $\mu=\mu_1+\tau \mu_1$. Then Theorems \ref{thm:pbw:Intr} and \ref{thm:GKLO:Intr} (on PBW basis and iGKLO representations) still hold. We restrict ourselves to even $\mu$ to match the geometric consideration in the sequel \cite{LWW25islice}. It is also possible to modify the affine Grassmannian setting to match with this variation, but we have refrained from taking this route.
\end{rem}

Since $\tau$ is an involution, it follows from \eqref{bbtau} that 
\beq
H_i^{(r)}=(-1)^{r}H_{\tau i}^{(r)},
\eeq
and in particular,
\beq
H_i^{(2r+1)}=0, \qquad \text{ if } \tau i=i.
\eeq

\begin{lem}[{\cite[Lemma~3.4]{LZ24}}]\label{typeArel}
If $c_{\tau i,j}=0$, then the relation \eqref{hbNqs} is equivalent to 
\beq\label{alt}
[H_{i,r+1},B_{j,s} ]-[H_{i,r},B_{j,s+1} ]= \frac{c_{ij}}{2}[H_{i,r},B_{j,s}]_+.
\eeq    
\end{lem}

\begin{lem}\label{lem-reduction}
Assume that relations \eqref{def0}--\eqref{bbNqs} hold in an algebra with the same set of generators as $\Yt_{\mu}$. 
\begin{enumerate}
\item If \eqref{SerreIII} holds only for the special case $s=s_1=1$ and all $s_2>0 $, then \eqref{SerreIII} holds for all $s,s_1,s_2\in\bZ_{>0}$. 
\item If \eqref{SerreIII2} holds for the special case $s=s_1=s_2=1$, then  \eqref{SerreIII2} holds for all $s,s_1,s_2\in\bZ_{>0}$. 
\end{enumerate}
\end{lem}

\begin{proof}
Part (1) follows by the same argument of \cite[\S4.3]{LWZ25} while Part (2) follows by the same argument of \cite[Proposition~3.12]{LZ24}.
\end{proof}

\begin{dfn} \label{def:doubled}
The Cartan doubled iYangian $\Yt_\infty  $ of quasi-split type is the $\bC$-algebra with generators $H_{i}^{(r)}$, $ B_{i}^{(s)}$, for $i\in \mathbb I$, $r\in\bZ$, and $s\in\bZ_{>0}$ subject only to the relations \eqref{hhIII}--\eqref{SerreIII} (i.e., excluding \eqref{def0}).
\end{dfn}


Let $Q$ be the root lattice for $\g$. Consider the abelian group (called the $\imath$root lattice)
    \begin{align}  \label{Qtau}
    \Qt = Q / \langle \beta+\tau \beta \mid \beta \in Q\rangle.
    \end{align}
For $\beta \in Q$, denote its image by $\overline{\beta} \in \Qt$. The algebra $\Yt_\mu$ admits a grading by $\Qt$ defined by assigning degrees $\deg H_i^{(r)} = \overline{0}$ and $\deg B_i^{(s)} = \overline{\alpha_i}$.  Indeed, it is not hard to see that the defining relations of $\Yt_\mu$ are all homogeneous.  Let $T = \operatorname{Spec} \bC[Q]$ denote the torus whose character lattice is $Q$.  Then the grading on $\Yt_\mu$ by $\Qt$ corresponds to an action of a subgroup $\Tt $ of $T$:
    \begin{align} \label{eq:def of Tt}
    \Tt = \{ t \in T \mid \tau(t) = t^{-1}\}  = \operatorname{Spec} \bC[\Qt].
    \end{align}
Since $\tau$ permutes the basis $\{\alpha_i\}_{i \in \I}$ for  $Q$, we can identify $\Qt \cong \bZ^{\I_1} \times (\bZ/2\bZ)^{\I_0}$, where the $\bZ^{\I_1}$-factor has basis $\{\overline{\alpha_i}\}_{i \in \I_1}$ while $(\bZ/2\bZ)^{\I_0}$ is generated by  $\{\overline{\alpha_i} \}_{i \in \I_0}$.  If we quotient $\Qt$ by its torsion subgroup $(\bZ/2\bZ)^{\I_0}$, we obtain an induced grading on $\Yt_\mu$ by $\bZ^{\I_1}$. This corresponds to the adjoint action of the elements  $\{H_i^{(-\langle \mu, \alpha_i\rangle+1)} 
\mid i \in \I_{1}\}$.

\begin{lem}\label{lem:shiftqs}
Let $\mu,\nu$ be coweights such that both $\mu$ and $\nu+\tau \nu$ are even spherical. Suppose further that $\nu$ is anti-dominant. Then there exists a homomorphism 
\begin{align}
    \label{shift_hom}
    \iota^\tau_{\mu,\nu}:\Yt_\mu \longrightarrow \Yt_{\mu+\nu+\tau \nu}
\end{align}
defined by
\[
H_i^{(r)}\mapsto  H_i^{(r- \langle\nu+\tau \nu,\alpha_i\rangle )}, \qquad B_i^{(s)}\mapsto \begin{cases}B_i^{(s-\langle\nu,\alpha_i\rangle)}, &\text{ if }\langle\nu,\alpha_i\rangle \text{ is even},\\
\sqrt{-1}B_i^{(s-\langle\nu,\alpha_i\rangle)}, &\text{ if }\langle\nu,\alpha_i\rangle \text{ is odd},
\end{cases}
\]
for $r\in\bZ$ and $s\in \bZ_{>0}$.
\end{lem}

In particular, if $\tau=\mathrm{id}$, we call $\Yt_\mu$ 
 the \textit{shifted iYangian of split type}.
 
\begin{rem}\label{rem1}
If $\mu=0$, then $\Yt:=\Yt_0$ is exactly the twisted Yangian introduced in \cite{LWZ25degen} (for $\tau =\id$) and in \cite{LZ24} (for $\tau\neq \id$) via degeneration of Drinfeld presentations of affine $\imath$quantum groups of split and quasi-split types (cf. \cite{LW21, mLWZ24}). The definition of $\Yt_\mu$ formally makes sense for simply-laced generalized Cartan matrix $C$ as suggested in \cite{LWZ25}.

If $\mu$ is dominant and spherical, i.e., $\mu=\mu_1+\tau\mu_1$ for some coweight $\mu_1$, then we have the homomorphism
\[
\iota^\tau_{\mu,-\mu_1}:\Yt_{\mu} \longrightarrow \Yt.
\]
By the PBW theorem established below, the homomorphism $\iota^\tau_{\mu,-\mu_1}$ is injective and hence identifies $\Yt_{\mu} $ as a subalgebra of $\Yt$; see Theorem \ref{thm:pbw_arb}. 
\end{rem}

For $i\in \I$, we set
\begin{align*}
    B_i(u) =\sum_{r>0} B_i^{(r)} u^{-r}. 
\end{align*}

\begin{rem}\label{rem:a3}
For fixed $i\in \I$ such that $c_{i,\tau i}=0$, the subalgebra of $\Yt_\mu$ generated by  $H_{i}^{(r)},B_i^{(s)},B_{\tau i}^{(s)}$, $r\in \bZ$ and $s\in \bZ_{>0}$ is isomorphic (using Lemma \ref{typeArel} and Theorem \ref{thm:pbw_arb}) to a shifted Yangian for $\fksl_2$. Specifically, the identification is given as follows,
\[
e(u)\to B_i(u),\quad f(u)\to B_{\tau i}(-u),\quad h(u)\to H_i(u).
\]
If $i=\tau i$, then $H_{i}^{(r)},B_i^{(s)}$, $r\in \bZ$ and $s\in \bZ_{>0}$ generate a subalgebra of $\Yt_\mu$ that is isomorphic to a shifted iYangian of split type A of rank 1.
\end{rem}

For split type, Definition~\ref{def:qsplit} greatly simplifies and it is convenient to list it separately as follows. 

\begin{dfn}\label{def1}
The \textit{shifted iYangian} $\Yt_\mu= \Yt_\mu(\g)$ \textit{of split type} is the $\bC$-algebra with generators $H_{i}^{(r)}$, $ B_{i}^{(s)}$, for $i\in \mathbb I$, $r\in\bZ$, and $s\in\bZ_{>0}$, subject to the following relations, for $r,r_1,r_2\in \bZ$ and $s,s_1,s_2 \in \bZ_{>0}$:
\begin{align}
H_{i}^{(p)}=0~\text{ for }~p&<-\langle \mu,\alpha_i\rangle~\text{ and } ~H_{i}^{(-\langle \mu,\alpha_i\rangle)}=1,  \label{def}\\
[ H_{i}^{(r_1)}, H_{j}^{(r_2)}]&=0, \qquad { H_{i}^{(2r+1)}=0, } \label{hhN}\\    
[ H_i^{(r+2)}, B_{j}^{(s)}]-[H_i^{(r)}, B_{j}^{(s+2)}]&=c_{ij}[ H_{i}^{(r)}, B_{j}^{(s+1)}]_+ +\frac{1}{4}c_{ij}^2[ H_{i}^{(r)}, B_{j}^{(s)}],\label{hbN}\\
[ B_{i}^{(s_1+1)}, B_{j}^{(s_2)}]-[ B_{i}^{(s_1)}, B_{j}^{(s_2+1)}]&=\frac{1}{2}c_{ij} [B_{i}^{(s_1)}, B_{j}^{(s_2)}]_+ +2\delta_{ij}(-1)^{s_1} H_{i}^{(s_1+s_2)},\label{bbN}\\
[B_{i}^{(s_1)}, B_{j}^{(s_2)}]&=0,\hskip 4.4cm  \text{ if }c_{ij}=0,\label{bbN2}\\
\mathrm{Sym}_{s_1,s_2}\big[B_{i}^{(s_1)},[B_{i}^{(s_2)},B_{j}^{(s)}]\big]&= (-1)^{s_1-1}[H_{i}^{(s_1+s_2)},B_j^{(s-1)}],\hskip 1cm \text{ if } c_{ij}=-1. \label{serreN}
\end{align}
\end{dfn}
Note that $\Yt_\mu$ has a generating set 
\beq\label{rvset3}
\begin{split}
\big\{B_\beta^{(r)}| \beta\in\Delta^+, r>0\big\} \cup\big\{H_i^{(2p)}|i\in \I_0, 2p>-\langle \mu,\alpha_i\rangle\big\} \cup\big\{H_i^{(p)}|i\in \I_1, p>-\langle \mu,\alpha_i\rangle\big\}.    
\end{split}
\eeq

\begin{lem}\label{lem:shift}
Let $\mu$ be an even coweight and $\nu$ be an anti-dominant coweight. Then there exists a homomorphism $\iota_{\mu,\nu}^\tau:\Yt_\mu \to\Yt_{\mu+2\nu}$ defined by
\[
H_i^{(r)}\mapsto  H_i^{(r-2 \langle\nu,\alpha_i\rangle )}, \qquad B_i^{(s)}\mapsto \begin{cases}B_i^{(s-\langle\nu,\alpha_i\rangle)}, & \text{ if }\langle\nu,\alpha_i\rangle \text{ is even},\\
\sqrt{-1} B_i^{(s-\langle\nu,\alpha_i\rangle)}, & \text{ if }\langle\nu,\alpha_i\rangle \text{ is odd},
\end{cases}
\]
where $r\in\bZ$ and $s\in \bZ_{>0}$.
\end{lem}

\begin{rem}
For $\mu=0$, $\Yt_0$ is exactly the twisted Yangian introduced in \cite{LWZ25degen} via degeneration of affine $\imath$quantum groups of split types. Similar to Remark \ref{rem1}, if $\mu$ is dominant, then $\Yt_\mu$ can be identified as a subalgebra of $\Yt_0$. 
\end{rem}

\subsection{PBW basis}
\label{ssec:PBW for sty}

Consider a positive root $\beta$ and fix an arbitrary ordered decomposition $\beta=\alpha_{i_1}+\ldots+\alpha_{i_\ell}$ into simple roots such that the elements $[e_{i_1},[e_{i_2},\cdots[e_{i_{\ell-1}},e_{i_\ell}]\cdots]]$ is a nonzero element in the root subspace $\g_\beta$. For any $r>0$, we define a root vector in $\mc Y_\mu(\g)$
\be
B_\beta^{(r)}:=\Big[B_{i_1}^{(r)},\big[B_{i_2}^{(1)},\cdots[B_{i_{\ell-1}}^{(1)},B_{i_\ell}^{(1)}]\cdots\big]\Big].
\ee

\begin{prop}\label{prop:span}
The shifted iYangian $\Yt_\mu$ is spanned by ordered monomials in the elements \eqref{rvset3}.
\end{prop}

\begin{proof}
For the simplicity of notation, in this proof, we consider an order with respect to the subsets \eqref{rvset3}, i.e. the elements $H_i^{(s)}$ are always to the right of the elements $B_\beta^{(r)}$ in the ordered monomials.

By \eqref{hbNqs}, it is easy to see by induction on $r$ that $H_i^{(r)}B_j^{(s)}$ can be rewritten as $H_i^{(r)}B_j^{(s)}=
B_j^{(s)}H_i^{(r)}+\sum_{p,p':\ p+p'< r+s} m_{p,p'}B_{j}^{(p')}H_i^{(p)},$ for $m_{p,p'}\in\bC.$
Thus the algebra $\Yt_\mu$ is spanned by elements of the form
$B_{j_1}^{(n_1)}B_{j_1}^{(n_2)}\cdots B_{j_k}^{(n_k)}H_{i_1}^{(m_1)}H_{i_2}^{(m_2)}\cdots H_{i_\ell}^{(m_\ell)}.$

Next, we define a filtration on $\Yt_\mu$ by setting $\deg B_j^{(s)}=1$ and $\deg H_i^{(r)}=0$. Let $\gr\, \Yt_\mu$ be the associated grade and $\overline{B}_j^{(s)}$ (resp. $\overline{B}_\beta^{(r)}$) the corresponding image of ${B}_j^{(s)}$ (resp. ${B}_\beta^{(r)}$). Let $\gr\, \Yt^{>0}_\mu $ be the subalgebra of $\gr\, \Yt_\mu$ generated by $\overline{B}_j^{(s)}$. One can argue as in \cite[proof of Theorem 4.2]{Lu23} that the ordered monomials in $\overline          {B}_\beta^{(r)}$ span $\gr\, \Yt^{>0}_\mu$. This is because $\overline{B}_j^{(s)}$ satisfy the relations for the positive half of the corresponding ordinary Yangians and hence the claims follows from the corresponding result of the ordinary Yangians, see \cite{Lev93}. Altogether, it proves the proposition.
\end{proof}

\begin{rem} \label{pbwlz}
For $\mu=0$, the set of ordered monomials in the elements \eqref{rvset3} forms a basis
for $\Yt_0$, see \cite[Theorem 4.12]{LWZ25degen} and \cite[Theorem 3.6]{LZ24}.
\end{rem}

\begin{thm}  \label{thm:pbwqs}
Let $\mu$ be even spherical and anti-dominant. Then the set of ordered monomials in the elements \eqref{rvset3} forms a basis for $\Yt_\mu$.
\end{thm}
The proof of Theorem~\ref{thm:pbwqs} will be given in \S\ref{subsec:proofPBW} below. We can strengthen Theorem~\ref{thm:pbwqs} as follows. 

\begin{thm}  [PBW basis]\label{thm:pbw_arb}
Let $\mu$ be an arbitrary even spherical coweight. 
\begin{enumerate}
\item 
The set of ordered monomials in the root vectors \eqref{rvset3} forms a basis for $\Yt_\mu$. 
\item 
For any anti-dominant $\nu$ such that $\nu+\tau \nu$ is even spherical, the shift homomorphism $\iota^\tau_{\mu,\nu}:\Yt_\mu\to \Yt_{\mu+\nu+\tau \nu}$  in \eqref{shift_hom} is injective.
\end{enumerate}
\end{thm}

\begin{proof}
By Lemma \ref{lem:shiftqs}, we have a shift homomorphism $\iota^\tau_{\mu,\nu}:\Yt_\mu \to \Yt_{\mu+\nu+\tau \nu} $. We pick $\nu$ so that $\mu+\nu+\tau \nu$ is anti-dominant. Then the root vectors of $\Yt_\mu$ are sent to root vectors (up to a scalar multiple) of $\Yt_{\mu+\nu+\tau \nu}$ under the shift homomorphism $\iota^\tau_{\mu,\nu}$. By Proposition \ref{prop:span}, the set of ordered monomials in the root vectors of $\Yt_\mu$ is a spanning set of $\Yt_\mu$. By Theorem \ref{thm:pbwqs}, the set is bijectively sent to a set of linearly independent ordered monomials (in the root vectors) in $\Yt_{\mu+\nu+\tau \nu}$. Hence this spanning set is also linearly independent and thus form a basis of $\Yt_\mu$, proving the first statement. The second statement follows from the first statement.
\end{proof}


\subsection{Proof of PBW basis Theorem \ref{thm:pbwqs}}
\label{subsec:proofPBW}

We shall follow the strategy of \cite[\S 3.12]{FKPRW}. Recall $\Yt_\infty$ from Definition \ref{def:doubled}.
\begin{dfn}
The algebra ${}^\imath\wtl\Y$ is the quotient of $\Yt_\infty$ by the relations $H_i^{(p)}=0$ for all $i\in \I$ and $p<0$.    
\end{dfn}
We use the notation $\wtl H_i^{(r)}$ and $\wtl B_i^{(s)}$ (also $\wtl B_\beta^{(s)}$)  for the generators of ${}^\imath\wtl \Y$.

\begin{lem}
The ordered monomials in the elements of the set
\beq\label{pbwpf1}
\big\{\wtl B_\beta^{(r)}\mid \beta\in\Delta^+, r>0\big\}\cup\big\{\wtl H_i^{(2p)}\mid i\in \I_0, p\gge 0\big\}\cup\big\{\wtl H_i^{(p)}\mid i\in \I_1, p\gge 0\big\},
\eeq
form a basis for ${}^\imath\wtl \Y$.
\end{lem}
\begin{proof}
Clearly, the elements $\wtl H_i^{(0)}$ for $i\in \I$ are central in ${}^\imath\wtl \Y$. One checks that the map 
\[
{}^\imath\wtl \Y \to \Yt\otimes_{\bC}\bC[\xi_i\mid i\in \I],\quad \text{defined by}\quad \wtl H_i^{(r)}\mapsto H_i^{(r)}\otimes \xi_i\xi_{\tau i},\quad \wtl B_i^{(r)}\mapsto B_i^{(r)}\otimes \xi_i,
\]
induces an algebra homomorphism. Then proceed as Proposition \ref{prop:span}, one finds that the ordered monomials in the elements of  \eqref{pbwpf1} span ${}^\imath\wtl\Y$. Finally, similar to Theorem \ref{thm:pbw_arb}, we show that this spanning set is bijectively sent to a subset of a basis for 
$\Yt\otimes_{\bC}\bC[\xi_i:i\in \I]$. Here we also need the PBW theorem for $\Yt$ (see Remark~\ref{pbwlz}). Thus the spanning set is linearly independent and hence a basis.
\end{proof}

\begin{cor}\label{cortl}
If $\mu$ is anti-dominant, then ${}^\imath\wtl \Y$ is free as a right module over the polynomial ring 
\beq\label{ringR}
\mathcal R:=\bC[\wtl H_i^{(2p)},\wtl H_{j}^{(q)}\mid i\in \I_0,0\lle 2p\lle -\langle \mu,\alpha_i\rangle,j\in \I_1,0\lle q\lle -\langle \mu,\alpha_j\rangle]
\eeq 
with a basis given by ordered monomials in the elements of the set
\[
\big\{\wtl B_\beta^{(r)}\mid \beta\in\Delta^+, r>0\big\}\cup
\big\{\wtl H_i^{(2p)}\mid i\in \I_0, 2p>-\langle \mu,\alpha_i\rangle\big\}\cup
\big\{\wtl H_i^{(p)}\mid i\in \I_1, p>-\langle \mu,\alpha_i\rangle\big\}.
\]
\end{cor}

Suppose further that $\mu$ is anti-dominant, then there is a surjective homomorphism $\pi:{}^\imath\wtl \Y \twoheadrightarrow \Yt_\mu $ defined by
\[
\wtl H_i^{(r)}\to H_i^{(r)},\qquad \wtl B_i^{(s)}\to B_i^{(s)}.
\]
The kernel of the projection $\pi$ is the 2-sided ideal
\beq\label{Imu}
\mc I_\mu:=\big\langle \wtl H_i^{(2r)}- \delta_{2r,-\langle \mu,\alpha_i\rangle},\wtl H_j^{(p)}- \delta_{p,-\langle \mu,\alpha_j\rangle}\big\rangle_{\rm 2-sided},
\eeq
where $i\in \I_0$, $0\lle 2r\lle -\langle \mu,\alpha_i\rangle$, $j\in\I_1$, $0\lle p\lle -\langle \mu,\alpha_j\rangle$.

Also denote by $\mc I_\mu^{\rm left}$ the left ideal generated by the same elements in \eqref{Imu}.
\begin{lem}\label{lemleft}
We have $\mc I_\mu=\mc I^{\rm left}_\mu$.
\end{lem}
\begin{proof}
It is sufficient to prove that $\mc I^{\rm left}_\mu$ is also a right ideal.

Recall that in the algebra ${}^\imath\wtl \Y$, if $i\in \I_0$, then we have the following relations 
\[
\wtl H_i^{(2r-1)}=0,\qquad [\wtl H_i^{(0)},\wtl B_{j}^{(s)}]=0,\qquad [\wtl H_i^{(2)},\wtl B_{j}^{(s)}]=2c_{ij}B_j^{(s+1)}\wtl H_i^{(0)},
\]
\[
[\wtl H_i^{(r+2)},\wtl B_j^{(s)}]=[\wtl H_i^{(r)},\wtl B_j^{(s+2)}]+c_{ij}[\wtl H_i^{(r)},\wtl B_j^{(s+1)}]+\frac{1}{4}c_{ij}^2[\wtl H_i^{(r)},\wtl B_j^{(s)}]+2c_{ij}\wtl B_j^{(s+1)}\wtl H_i^{(r)}.
\]
Thus an obvious induction on $r$ shows that for any $r\gge 1$, we have
\[
[\wtl H_i^{(r)},\wtl B_j^{(s)}]\in \big\langle \wtl H_i^{(0)},\wtl H_i^{(1)},\ldots,\wtl H_i^{(r-1)}\big\rangle^{\rm left}.
\]
Note that here $\wtl H_i^{(p)}=0$ if $p$ is odd. Therefore, for $1\lle 2r\lle -\langle \mu,\alpha_i\rangle$, we have
\[
(\wtl H_i^{(2r)}- \delta_{2r,-\langle \mu,\alpha_i\rangle})\wtl B_j^{(s)}\in \wtl B_j^{(s)}(\wtl H_i^{(2r)}- \delta_{2r,-\langle \mu,\alpha_i\rangle})+\mc I_\mu^{\rm left}=\mc I_\mu^{\rm left}.
\]
A similar calculation also works for $i\in \I_1$. It implies that the right multiplication by $\wtl B_j^{(s)}$ preserves $\mc I_\mu^{\rm left}$. It is also clear that right multiplication by $\wtl H_j^{(r)}$ preserves $\mc I_\mu^{\rm left}$. Since the algebra ${}^\imath\wtl \Y $ is generated by $\wtl B_j^{(s)}$ and $\wtl H_j^{(r)}$, we conclude that $\mc I_\mu^{\rm left}$ is a right ideal.
\end{proof}

\begin{proof}[Proof of Theorem \ref{thm:pbwqs}]
Let $\Gamma:\mc R\rightarrow \bC$
be  the homomorphism defined by sending $\wtl H_i^{(2p)}-\delta_{2p,-\langle \mu,\alpha_i\rangle}$ and $\wtl H_{j}^{(q)}-\delta_{q,-\langle \mu,\alpha_j\rangle}$ to $0$, for $i\in \I_0$, $0\lle 2p\lle -\langle \mu,\alpha_i\rangle$, $j\in \I_1$, and $0\lle q\lle -\langle \mu,\alpha_j\rangle$. It follows from Lemma~\ref{lemleft} that $\Yt_\mu$ is the base change of the right module ${}^\imath\wtl \Y $ with respect to the homomorphism $\Gamma$,
\[
\Yt_\mu ={}^\imath\wtl\Y \otimes_{\mc R}\bC.
\]
By Corollary \ref{cortl}, ${}^\imath\wtl \Y $ is a free right module over $\mc R$. Therefore, the basis from Corollary \ref{cortl} gives rise to a basis for $\Yt_\mu$ over $\bC$, completing the proof.
\end{proof}

\section{$\mathrm{i}$GKLO representations of shifted iYangians}
\label{sec:GKLO}

In this section, we formulate a family of iGKLO representations of shifted iYangians of arbitrary quasi-split ADE type. 

\subsection{Ring of difference operators}
\label{ssec:quantumtorus}

Fix a dominant $\tau$-invariant coweight $\la$, i.e., $\tau \la =\la$,  and an even spherical coweight $\mu$ such that $\lambda \gge \mu$. We denote
\begin{align}\label{ell}
\la-\mu=\sum_{i\in \I}\bv_i\alpha_i^\vee,
\end{align}
where $\bv_i\in\bN$ for $i\in \I$. Denoting $\mathbb Z_2 =\{\bar 0, \bar 1\}$, we set
\beq \label{ell_theta}
\fkv_i =
\begin{cases} 
\bv_i, & \text{if } \tau i \neq i \\ 
\lfloor\tfrac{1}{2} \bv_i \rfloor, & \text{if } \tau i = i,
\end{cases}
\qquad 
\theta_i =
\begin{cases} 
0, & \text{if } \tau i \neq i \\ 
\delta_{\overline{\mathbf v}_i, \bar{1}}, & \text{if } \tau i = i. 
\end{cases}
\eeq
Introduce $\vartheta_i$ by
\beq\label{vartheta}
\vartheta_{i}=\begin{cases}
\theta_i, & \text{if }\tau i\ne i,\\
\max\{ \theta_j\mid j\in\I \text{ and }c_{ij}\ne 0\},& \text{if }\tau i= i.
\end{cases}
\eeq
Denote $\mathbf w_i=\langle \lambda,\alpha_i\rangle$. We also set
\beq
\label{varsigma}
\fkw_i =
\begin{cases} 
\bw_i, & \text{if } \tau i \neq i \\ 
\lfloor\tfrac{1}{2} \bw_i \rfloor, & \text{if } \tau i = i,
\end{cases}
\qquad 
\varsigma_i =
\begin{cases} 
0, & \text{if } \tau i \neq i \\ 
\delta_{\overline{\bw}_i, \bar{1}}, & \text{if } \tau i = i. 
\end{cases}
\eeq

Note that 
\[
\fkv_{\tau i}=\fkv_i,\qquad \fkw_{\tau i}=\fkw_i, \qquad \text{for all }i\in \I. 
\]
Recall $C=(c_{ij})$ is the Cartan matrix. Throughout the paper, we shall impose the following fundamental parity condition on the dimension vector $\bv =(\bv_i)_{i\in\I}$; see \eqref{ell}--\eqref{ell_theta}:
\beq\label{parity}
c_{ij}\theta_i\theta_j=0, \qquad \text{ for } i\neq j \in \I,
\eeq
that is, at least one of $\bv_i$ and $\bv_j$ is even when $i\neq j\in \I$ are connected. 
This condition is needed in this section and also turns out to be required for geometric constructions in \cite[Section~5]{LWW25islice}; see, e.g., \cite[Theorem~5.11]{LWW25islice}. 

\begin{rem}  \label{rem:parity}
  Let $i\in \I$. If $\theta_i=1$, then the evenness of $\mu$ and the parity condition \eqref{parity} imply that $\mathbf w_i =\langle \lambda,\alpha_i\rangle$ is even and hence $\varsigma_i=0$. 
\end{rem}

Let $\bm z:=(z_{i,s})_{i\in\iI,1\lle s\lle \fkw_i}$ be formal variables and denote the polynomial ring
\[
\bC[\bm z]=\bC[z_{i,s}]_{i\in\iI,1\lle s\lle \fkw_i}
\]
and define the new $\C$-algebra
\[
\Yt_\mu(\g)[\bm z]:=\Yt_\mu(\g)\otimes \bC[\bm z],
\]
with new central elements $z_{i,s}$.
Consider the $\bC$-algebra
\begin{align} \label{A:generators}
\mc A:=\bC[\bm z]\langle w_{i,r},\dfo_{i,r}^{\pm 1}, (w_{i,r}\pm w_{i,r'}+m)^{-1},(w_{i,r}+\hf m)^{-1}\rangle_{i \in \iI,1\lle r\ne r'\lle \fkv_i,m\in\bZ},
\end{align}
subject to the relations
\begin{align} \label{A:relations}
[ \dfo_{i,r}^{\pm 1},w_{j,r'}] =\pm\delta_{ij}\delta_{r,r'}\dfo_{i,r}^{\pm 1},\qquad [w_{i,r},w_{j,r'}]=[\dfo_{i,r},\dfo_{j,r'}]=0,\qquad \dfo_{i,r}^{\pm 1}\dfo_{i,r}^{\mp 1}=1. 
\end{align}
It is convenient to extend the notation $z_{i,s}$, $w_{i,r}$, and $\dfo_{i,r}$ to all $i\in \I$ as follows. First set 
\beq\label{signzw2}
z_{\tau i,s}=-z_{i,s}, \qquad \text{ for } i\in \I_1 \text{ and } 1\lle s\lle \fkw_i. 
\eeq
We further set
\begin{align}
 \label{signzw}
 w_{\tau i,r} :=-w_{i,\fkv_i+1-r},
\qquad
\dfo_{\tau i,r} :=\dfo_{i,\fkv_i+1-r}^{-1}, 
 \qquad \text{ for } i\in \I_1 \text{ and } 1\lle r\lle \fkv_i.
\end{align}

Given a monic polynomial $f(u)$ in $u$, we define
\begin{align}  \label{eq:f^-}
f^-(u):=(-1)^{\deg f}f(-u)
\end{align}
to be the monic polynomial whose roots are the opposite of the roots of $f(u)$.

For each $i\in \I_0$, define
\begin{align}  
W_i(u)&=\prod_{r=1}^{\fkv_i}(u- w_{i,r}),\qquad\quad~~~ Z_i(u)=\prod_{s=1}^{\fkw_i}(u-z_{i,s}),
\label{WiZi} \\
\W_i(u)&=u^{\theta_i}\prod_{r=1}^{\fkv_i}(u^2- w^2_{i,r}),\qquad \Z_i(u)=u^{\varsigma_i}\prod_{s=1}^{\fkw_i}(u^2-z^2_{i,s}).
\label{WiZibold}
\end{align}
Then we have $\deg \W_i(u) = \bv_i$ and $\W_i(u) = \W_i^-(u)$,  and similarly $\deg \Z_i(u) = \mathbf w_i$ and $\Z_i^-(u) = \Z_i(u)$. We also define 
\begin{align} \label{Wir}
\W_i^\circ(u)= \prod_{r=1}^{\fkv_i}(u^2- w^2_{i,r}),\quad\W_{i,r}(u)=u^{\theta_i}(u+w_{i,r})\prod_{s=1,s\ne r}^{\fkv_i}(u^2- w^2_{i,s}).
\end{align}
Introduce
\begin{align} \label{WiZibar}
\overline W^-_i(u):=u^{\theta_i}W^-_i(u),\qquad \overline Z^-_i(u):=u^{\varsigma_i}Z^-_i(u).
\end{align}
For simplicity, set
\beq \label{varkappa}
\varkappa(u)=1-\frac{1}{2u},\qquad \bm\varkappa(u)=1-\frac{1}{4u^2}.
\eeq

For each $i\in \I_1$, we fix a choice of $\zeta_i\in \bN$ such that $1\lle \zeta_i\lle \fkv_i$ and extend it to $i\in \I_1\cup\I_{-1}$ by 
\begin{align*} 
\zeta_{\tau i}=\fkv_i-\zeta_{i}.
\end{align*}

For $i\in \I_1\cup \I_{-1}$, set
\begin{align}
\label{WiZi_qs}
\begin{split}
W_i(u)=\prod_{r=1}^{\zeta_i}(u- w_{i,r}), &\qquad 
\W_i(u)=\prod_{r=1}^{\fkv_i}(u- w_{i,r}),
\\
\Z_i(u)=\prod_{s=1}^{\fkw_i}(u-z_{i,s}),
& \qquad
\W_{i,r}(u) =\prod_{s=1,s\ne r}^{\fkv_i}(u- w_{i,s}).
\end{split}
\end{align}
It follows by \eqref{signzw2} and \eqref{signzw} that for $i\in \I_1\cup \I_{-1}$ we have
\[
\W_i(u)=W_i(u)W_{\tau i}^-(u),\qquad 
\W_{\tau i}(u)=\W_i^-(u),\qquad 
\Z_{\tau i}(u)=\Z_i^-(u).
\]
We pick a monic polynomial $Z_i(u)$, for each $i\in \I_1\cup \I_{-1}$, such that
\[
\Z_i(u)=Z_i(u)Z_{\tau i}^-(u)=(-1)^{\deg Z_{\tau i}}Z_i(u)Z_{\tau i}(-u).
\]

\subsection{iGKLO representations for type AI}

It is convenient to first work with the shifted version of twisted Yangian of split type A (i.e., type AI), which corresponds to $\gl_n$ instead of $\mathfrak{sl}_n$.  We will use the Drinfeld type presentation established in \cite{LWZ25} via Gauss decomposition. In this case, $\I=\I_0=\{1,\ldots,n-1\}$.

Given a Laurent series $X(u)=\sum_{r\in \bZ}X^{(r)}u^{-r}$, denote by $(X(u))^{\star}$ its principal part:
\[
(X(u))^{\star}:=\sum_{r>0}X^{(r)}u^{-r}.
\]

\begin{dfn}
The \textit{shifted iYangian $\Yt_{\mu}(\gl_n)$ associated to an even coweight $\mu$} is the associative $\bC$-algebra generated by $D_i^{(r)}$, $\wtl D_i^{(\tl r)}$, and $E_j^{(s)}$, where $1\lle i\lle n$, $1\lle j<n$, $r\in \bZ_{\gge \langle  \mu,\ve_i\rangle}$, $\tl r\in\bZ_{\gge -\langle \mu,\ve_i\rangle}$ and $s\in\bZ_{>0}$, with the relations
\begin{align}
&D_{i}^{(\langle \ve_i,\mu\rangle)}=1,\qquad [D_i(u),D_j(v)]=0,\qquad D_i(u)\wtl D_{i}(u)=\wtl D_i(u)D_i(u)=1,\label{ddgl}\\
&\wtl D_i(u)D_{i+1}(u)=\wtl D_i(-u+i)D_{i+1}(-u+i),\label{deven}\\
&[D_i(u),E_i(v)]=\frac{D_i(u)(E_i(u)-E_i(v))}{u-v}+\frac{(E_i(v)-E_i(-u+i))D_i(u)}{u+v-i},\label{de}\\
&[E_i(v),D_{i+1}(u)]=\frac{D_{i+1}(u)(E_i(u)-E_i(v))}{u-v}+\frac{(E_i(v)-E_i(-u+i))D_{i+1}(u)}{u+v-i},\label{ed}\\
&[E_i(u),E_{i+1}(v)]=\frac{-E_i(u)E_{i+1}(v)+[E_i^{(1)},E_{i+1}(v)]-[E_i(u),E_{i+1}^{(1)}]}{u-v},\label{ee2}\\
&[E_i(u),E_j(v)]=0,\qquad \text{ if }c_{ij}=0,\label{ee=0}\\
&[E_i(u),E_i(v)]=-\frac{(E_i(u)-E_i(v))^2}{u-v}-\frac{(\wtl D_i(u)D_{i+1}(u))^\star-(\wtl D_i(v)D_{i+1}(v))^\star}{u+v-i},\label{ee}\\
&\big[B_i^{(1)},[B_i(u),B_j^{(1)}]\big]+\big[B_i(u),[B_i^{(1)},B_j^{(1)}]\big] \nonumber \\  &\hskip1.7cm=\big(B_j(-u+\tfrac12)H_i(u)-B_j(u+\tfrac12)H_i(u)\big)^\star, \quad \text{ if }c_{ij}=-1,\label{serregl}
\end{align}
where
\[
D_i(u)=\sum_{r\gge \langle \ve_i,\mu\rangle}D_i^{(r)}u^{-r},\quad \wtl D_i(u)=\sum_{r\gge -\langle \ve_i,\mu\rangle}\wtl D_i^{(r)}u^{-r},\quad E_j(u)=\sum_{s>0}E_j^{(s)}u^{-s},
\]
\begin{align}  \label{Bi:shiftEi}
\quad B_i(u)=E_i(u+\tfrac{i}2),\quad H_i(u)= \wtl D_{i}(u+\tfrac{i}2)D_{i+1}(u+\tfrac{i}2).
\end{align}
\end{dfn}

\begin{lem}\label{lem:sl-gl}
For any even coweight $\mu$, there exists a homomorphism
\[
\eta_{\mu}:\Yt_\mu(\mathfrak{sl}_n)\longrightarrow \Yt_{\mu}(\gl_n)
\]
defined by
\[
B_i(u)\mapsto B_i(u),\qquad H_i(u)\mapsto H_i(u).
\]
\end{lem}
\begin{proof}
It follows from the same calculation as in \cite[\S4]{LWZ25}. Note that the Serre relation \eqref{serregl} is formulated differently from \eqref{serreN}. Indeed, the Serre relation \eqref{serregl} corresponds to \cite[Lemma 6.5]{LPTTW25}, see \cite[Remark 4.3]{LPTTW25} for more detail.
\end{proof}

In order to describe the iGKLO representations for shifted iYangians of $\gl_n$ (cf. \cite[Theorem.~2.35]{FPT22} for shifted Yangians of $\gl_n$), we need additional notation. Recall the polynomials $\W_i(u)$ and $\Z_i(u)$ from \S\ref{ssec:quantumtorus} and note that $i=\tau i$ for all $i\in\I$. We denote
\[
\vartheta_0=0,\quad \fkv_0=0, \quad \fkv_n=0,\quad \W_0(u)=\W_{n}(u)=1,
\]
and take any
\[
\Z_0(u)=u^{-\theta_1}\prod_{x\in\bC}(u^2-x^2)^{m_x},
\]
that satisfies $m_x\in\bZ$, $m_x\ne 0$ for only finitely many $x\in\bC$, and 
$$\langle \la,\ve_1\rangle=\theta_1-2\sum_{x\in\bC}m_x.$$

\begin{thm}\label{thm:iGKLO:gl}
Let $\mu$ be an even coweight. There is a homomorphism $\Phi_{\mu}^\la:\Yt_\mu(\gl_n)[\bm z]\to \mc A$, defined by
\begin{align}
D_i(u)\mapsto& \frac{\W_i(u-\tfrac{i-1}{2})}{\W_{i-1}(u-\tfrac{i}2)}\prod_{j=0}^{i-1}\Big(\bm\varkappa(u-\tfrac{j}{2})^{\vartheta_{j}}\Z_j(u-\tfrac{j}2)\Big),
\notag \\
E_i(u)\mapsto& -\sum_{r=1}^{\fkv_i}\frac{\varkappa(w_{i,r})^{-\vartheta_i}\W_{i-1}(w_{i,r}-\tfrac12)\overline W^-_{i+1}(w_{i,r}-\tfrac12)}{(u-\tfrac{i-1}{2}-w_{i,r})\W_{i,r}(w_{i,r})}\dfo_{i,r}^{-1}
\notag \\
&~ -\sum_{r=1}^{\fkv_i}\frac{\varkappa(-w_{i,r})^{-\vartheta_i} \Z_i(w_{i,r}+\tfrac12) W_{i+1}(w_{i,r}+\tfrac12)}{(u-\tfrac{i-1}{2}+w_{i,r})\W_{i,r}(w_{i,r})}\dfo_{i,r}
 \label{eq:Eiu} \\
& +\sqrt{(-1)^{\fkw_i+\fkv_{i-1}+\fkv_{i+1}}} \frac{\theta_i Z_i(0)}{(u-\tfrac{i}{2})\W_i^\circ(\tfrac12)}\prod_{j\leftrightarrow i}W_{j}(0).
\notag
\end{align}
(The image of $\wtl D_i(u)$ is the inverse of the image of $D_i(u)$.)
\end{thm}
Theorem \ref{thm:iGKLO:gl} is proved in \S\ref{sec:proof1} below.

There is some flexibility in the formula for $E_i(u)$ in \eqref{eq:Eiu}, where the factor $\Z_i(w_{i,r}+\tfrac12)$ can be split into 2 factors, one for each of the first two summations. In this way, using notation  \eqref{WiZibold} and \eqref{WiZibar} (also see \eqref{eq:f^-}), the formula for $E_i(u)$ in \eqref{eq:Eiu} can be modified to be 
\begin{align*}
E_i(u)\mapsto &-\sum_{r=1}^{\fkv_i}\frac{\varkappa(w_{i,r})^{-\vartheta_i} Z_i(w_{i,r}-\tfrac12)\W_{i-1}(w_{i,r}-\tfrac12)\overline W^-_{i+1}(w_{i,r}-\tfrac12)}{(u-\tfrac{i-1}{2}-w_{i,r})\W_{i,r}(w_{i,r})}\dfo_{i,r}^{-1}
\\
&-\sum_{r=1}^{\fkv_i}\frac{\varkappa(-w_{i,r})^{-\vartheta_i} \overline Z_i^-(w_{i,r}+\tfrac12) W_{i+1}(w_{i,r}+\tfrac12)}{(u-\tfrac{i-1}{2}+w_{i,r})\W_{i,r}(w_{i,r})}\dfo_{i,r}
\\
&+\sqrt{(-1)^{\fkw_i+\fkv_{i-1}+\fkv_{i+1}}} \frac{\theta_i Z_i(0)}{(u-\tfrac{i}{2})\W_i^\circ(\tfrac12)}\prod_{j\leftrightarrow i}W_{j}(0).
\end{align*}

Fix an arbitrary orientation of the diagram $\I$. For $i,j\in \I$, we denote by $j\leftrightarrow i$ if there is an arrow $j\leftarrow i$ or $j\rightarrow i$. We give another version of iGKLO with the assumption that $\theta_i=0$ for $i$ even (recall the parity assumption \eqref{parity}), whose formulas are more similar to the traditional Gelfand-Zeitlin formulas. The following theorem is not used elsewhere in this paper.

\begin{thm}\label{thm:iGKLO:gl-alt}
Let $\mu$ be an even coweight and suppose $\theta_i=0$ for $i$ even. Then there is a homomorphism $\Phi_{\mu}^\la:\Yt_\mu(\gl_n)[\bm z]\to \mc A$, defined by
\begin{align*}
D_i(u)\mapsto& \frac{\W_i(u-\tfrac{i-1}{2})}{\W_{i-1}(u-\tfrac{i}2)}\prod_{j=0}^{i-1}\Big(\bm\varkappa(u-\tfrac{j}{2})^{\vartheta_{j}}\Z_j(u-\tfrac{j}2)\Big),\\
E_i(u)\mapsto& -\sum_{r=1}^{\fkv_i}\frac{\varkappa(w_{i,r})^{-\vartheta_i}Z_i(w_{i,r}-\tfrac12)}{(u-\tfrac{i-1}{2}-w_{i,r})\W_{i,r}(w_{i,r})}\prod_{j\rightarrow i}W_{j}(w_{i,r}-\tfrac12)\prod_{j\leftarrow i}W_{j}^-(w_{i,r}-\tfrac12)\dfo_{i,r}^{-1}\\&~ -\sum_{r=1}^{\fkv_i}\frac{\varkappa(-w_{i,r})^{-\vartheta_i}\overline Z_i^-(w_{i,r}+\tfrac12)}{(u-\tfrac{i-1}{2}+w_{i,r})\W_{i,r}(w_{i,r})}\prod_{j\rightarrow i}W_{j}^-(w_{i,r}+\tfrac12)\prod_{j\leftarrow i}W_{j}(w_{i,r}+\tfrac12)\dfo_{i,r}\\
& +\sqrt{(-1)^{\fkw_i+\fkv_{i-1}+\fkv_{i+1}}} \frac{\theta_i Z_i(0)}{(u-\tfrac{i}{2})\W_i^\circ(\tfrac12)}\prod_{j\leftrightarrow i}W_{j}(0),\qquad \text{for $i$ odd},\\
E_i(u)\mapsto& -\sum_{r=1}^{\fkv_i}\frac{\varkappa(w_{i,r})^{-\vartheta_i}Z_i(w_{i,r}-\tfrac12)}{(u-\tfrac{i-1}{2}-w_{i,r})\W_{i,r}(w_{i,r})}\prod_{j\rightarrow i}\W_{j}(w_{i,r}-\tfrac12)\dfo_{i,r}^{-1}\\&~ -\sum_{r=1}^{\fkv_i}\frac{\varkappa(-w_{i,r})^{-\vartheta_i}\overline Z_i^-(w_{i,r}+\tfrac12)}{(u-\tfrac{i-1}{2}+w_{i,r})\W_{i,r}(w_{i,r})}\prod_{j\leftarrow i}\W_{j}(w_{i,r}+\tfrac12)\dfo_{i,r},\qquad \text{for $i$ even}.
\end{align*}
(The image of $\wtl D_i(u)$ is the inverse of the image of $D_i(u)$.)
\end{thm}

\begin{proof}
The proof is similar to that of Theorem \ref{thm:iGKLO:gl}. The key point is to check that Lemma \ref{lem:chi} still holds for the new formulas.
\end{proof}

\subsection{iGKLO representations for quasi-split type}
\label{ssec:GKLOqs}

Recall the polynomials $\W_i(u)$ and $\Z_i(u)$ from \S\ref{ssec:quantumtorus}.
Fix an arbitrary orientation of the diagram $\I$ such that for each $i\in \I$ with $i\ne \tau i$, if $i\to j$, then $\tau j\to \tau i$, or if $j\to i$, then $\tau i\to \tau j$. Let
\beq\label{s-def}
\wp_{i}=\begin{cases}
1, & \text{if }~i\leftarrow \tau i,\\
-1, & \text{if }~i\rightarrow \tau i,\\
0, & \text{if }~c_{i,\tau i}=0, 2.
\end{cases}
\eeq

We present our main result in this section on the iGKLO representations for shifted iYangians $\Yt_\mu$ of arbitrary quasi-split ADE types. 

\begin{thm} \label{thm:GKLOquasisplit}
Let $(\I,\tau)$ be any quasi-split Satake diagram. 
Let $\la$ be a dominant $\tau$-invariant coweight and $\mu$ be an even spherical coweight subject to the constraint \eqref{parity} such that $\la\gge \mu$. 
Then there exists a homomorphism 
\[
\Phi_{\mu}^\la:\Yt_\mu[\bm z]\longrightarrow \mc A
\]
such that (see \eqref{ell}--\eqref{varsigma} and \eqref{WiZi}--\eqref{WiZi_qs} for notations)
\begin{align*}
H_i(u)&\mapsto \Big(1+\frac{\wp_i}{4u}\Big) \frac{\bm\varkappa(u)^{\vartheta_i}\Z_i(u)}{\W_{i}(u-\tfrac{1}2)\W_{i}(u+\tfrac{1}2)}\prod_{j\leftrightarrow i}\W_j(u),\qquad  \text{ for }\ i\in \I,
\\
B_i(u)&\mapsto -\sum_{r=1}^{\zeta_i}\frac{Z_i(w_{i,r}-\tfrac12)}{(u+\tfrac{1}{2}-w_{i,r})\W_{i,r}(w_{i,r})}\prod_{j\rightarrow i}\W_{j}(w_{i,r}-\tfrac12)\eth_{i,r}^{-1}\\
&\quad\  - \sum_{r=1}^{\zeta_{\tau i}}\frac{Z_i^-(w_{\tau i,r}+\tfrac12)}{(u+\tfrac{1}{2}+w_{\tau i,r})\W_{\tau i,r}( w_{\tau i,r})}\prod_{\tau j\leftarrow \tau i}\W_{\tau j}(w_{\tau i,r}+\tfrac12)\eth_{\tau i,r},
\qquad \text{ for }\  i\in \I\setminus\I_0,
\end{align*}
and 
\begin{align*}
B_i(u)&\mapsto -\sum_{r=1}^{\fkv_i}\frac{\varkappa(w_{i,r})^{-\vartheta_i}Z_i(w_{i,r}-\tfrac12)}{(u+\tfrac{1}{2}-w_{i,r})\W_{i,r}(w_{i,r})}\prod_{j\rightarrow i}\W_{j}(w_{i,r}-\tfrac12)\prod_{\substack{j\leftarrow i\\j\in\I_0}}\overline W_{j}^-(w_{i,r}-\tfrac12)\dfo_{i,r}^{-1}\\
&~\quad- \sum_{r=1}^{\fkv_i}\frac{\varkappa(-w_{i,r})^{-\vartheta_i}\overline Z_i^-(w_{i,r}+\tfrac12)}{(u+\tfrac{1}{2}+w_{i,r})\W_{i,r}(w_{i,r})}\prod_{\substack{j\rightarrow i\\ j\in \I_{\pm 1}}}\W_{j}(w_{i,r}+\tfrac12)\prod_{\substack{j\leftarrow i\\ j\in\I_0}}W_{j}(w_{i,r}+\tfrac12)\dfo_{i,r}
\\
&~\quad+\sqrt{(-1)^{\fkw_i+\sum_{i\leftrightarrow j\in \iI}\fkv_j}} \frac{\theta_i Z_i(0)}{u\W_i^\circ(\tfrac12)}\prod_{j\leftrightarrow i}W_{j}(0), \qquad\qquad\qquad \text{ for }\ i\in \I_0.
\end{align*}
\end{thm}

The proof of Theorem \ref{thm:GKLOquasisplit} excluding the quasi-split type A$_{2n}$ in \S\ref{ssec:qspf} is similar to Theorem \ref{thm:iGKLO:gl}. Indeed, the verification of relations in $\Yt_\mu$ under $\Phi_{\mu}^\la$ is either similar to that of the ordinary shifted Yangians or shifted iYangians of split type, see Remark \ref{rem:a3}.

For the quasi-split type A$_{2n}$, all relations except for the Serre relation \eqref{SerreIII2} can be verified in the same way as for other types. We shall give detailed verifications for the relation \eqref{bbNqs} when $i=n$ and $j=n+1$ and the Serre relation \eqref{SerreIII2} in \S\ref{ssec:qsAIII2n}. Note that the extra factor $1+\frac{\wp_i}{4u}$ for $H_i(u)$ in the theorem above is compatible with \cite[(6.24)]{LZ24}.

\begin{rem}
It is possible to also formulate analogous iGKLO representations for shifted iYangians of split BCFG type, cf. \cite{LWZ25degen}. Proving it would require a better understanding of the complicated Serre relations.
\end{rem}

\subsection{Truncated shifted iYangians}

Inspired by the definition of ordinary truncated shifted Yangians \cite[\S B(viii)]{BFN19} and \cite[\S 4.3]{KWWY14}, we make the following.
\begin{dfn} \label{dfn:truncated stY}
Let $\la$ be a dominant $\tau$-invariant coweight and $\mu$ be an even spherical coweight with $\lambda \gge \mu$. The truncated shifted iYangian (TSTY), denoted $\Yt_\mu^\la$, is the $\C$-algebra given by the image of the iGKLO homomorphism $\Phi_{\mu}^\la:\Yt_\mu[\bm z]\rightarrow \mc A$. 
\end{dfn}

Define a ``Cartan" series $\mathsf{A}_i(u)$ in $\Yt_\mu[\bm z][\![u^{-1}]\!]$, for $i\in \I$, by
\beq\label{GKLO-A}
H_i(u)=\Big(1+\frac{\wp_{i}}{4u}\Big)\frac{\bm\varkappa(u)^{\vartheta_i}\Z_i(u)\prod_{j\leftrightarrow i}u^{\bv_j}}{(u^2-\frac14)^{\bv_i}}\frac{\prod_{j\leftrightarrow i} \mathsf A_j(u)}{\mathsf A_i(u-\frac12)\mathsf A_i(u+\frac12)},
\eeq
where $\wp_i$ is defined in \eqref{s-def}.
Expanding the series $\mathsf A_i(u)$ gives us a family of GKLO-type ``Cartan" elements $\mathsf{A}_i^{(r)}$ in $\Yt_\mu[\bm z]$, for $r>0, i\in \I$:
\beq\label{A-coeff}
\mathsf A_i(u)=1+\sum_{r>0}\mathsf A_i^{(r)}u^{-r}.
\eeq
Then we have the following simple lemma.
\begin{lem}
\label{lem:AiGKLO}
The iGKLO homomorphism $\Phi_\mu^\la$ from Theorem \ref{thm:GKLOquasisplit} sends
\[
\mathsf A_i(u)\mapsto u^{-\bv_i}\W_i(u).
\]    
\end{lem}

Define elements $\mathsf B_{i}^{(r)}$, for $i\in\I, r>0$, by 
\beq\label{GKLO-B}
\mathsf B_i(u):=u^{\bv_i-\theta_i}(u+\hf)^{\theta_i}B_i(u+\tfrac12)\mathsf A_i(u)=u^{\bv_i}\sum_{r>0}\mathsf B_{i}^{(r)}u^{-r},
\eeq
cf. \cite[(11.5)]{LPTTW25}. It is not hard to see that $\mathsf A_i^{(r)}$ and $\mathsf B_i^{(r)}$ for $i\in \I,r>\bv_i$ belong to the kernel of the homomorphism $\Phi_\mu^\la$. Motivated by the case of ordinary truncated shifted Yangians \cite[Remark B.21]{BFN19} and \cite[Theorem 4.10]{KWWY14}, we propose the following.
\begin{conj}  \label{conj:presentation}
 There is an isomorphism
$$
\Yt_\mu^\lambda \cong \Yt_{\mu} [\bm z] / \langle \mathsf A_i^{(r)},~\delta_{\bar{\bv}_j, \bar 0}\mathsf B_{j}^{(s)} \mid i \in \I,j\in\I_0, r > \bv_i,s>\bv_j \rangle,
$$
induced by the epimorphism $\Phi_\mu^\lambda$.
\end{conj}
\begin{rem}
    In the analogous (conjectural) presentation for ordinary truncated shifted Yangians, no elements like $\mathsf B_j^{(r)}$ are needed.  In our present twisted context, since the series $\mathsf A_i(u)$ is always even for $i \in \I_0$, we need to quotient by extra elements $\mathsf B_{i}^{(r)}$ for $r>\bv_i$ if $i\in \I_0$ and $\bv_i$ is even, cf. \cite[(11.6)]{LPTTW25}.
\end{rem}


Consider the subalgebra of $\Yt_\mu^\la$ generated over $\bC[\bm z]$ by the coefficients of all of the series $\mathsf A_i(u)$. (Equivalently, this subalgebra is generated over $\bC[\bm z]$ by the coefficients of the series $H_i(u)$.)  Using Lemma \ref{lem:AiGKLO}, one sees that this commutative subalgebra is a polynomial ring, having the following algebraically independent generators over $\bC[ \bm z]$:
\begin{equation}
    \label{eq:qis}
    \{ \mathsf A_i^{(2r)} : i \in \I_0, ~ 1 \lle 2r \lle \bv_i \} \cup \{\mathsf A_i^{(r)} : i \in \I_1, 1 \lle r \lle \bv_i \}.
\end{equation}
We call this the \emph{Gelfand-Tsetlin subalgebra} of $\Yt_\mu^\la$, and in many cases it is a maximal commutative subalgebra of $\Yt_\mu^\la$. We expect that it will be interesting to study corresponding categories of Gelfand-Tsetlin modules, similarly to \cite{Web24}. We note, however, that the algebras $\Yt_\mu^\la$ do not obviously fit into the context of \cite{Web24}, and in particular it is not clear whether $\Yt_\mu^\la$ is generally a Galois order in the sense of Futorny-Ovsienko \cite{FO10}.

Note that $\bC[\bm z]$ is a central subalgebra of $\Yt_\mu^\la$.  

\begin{prop}  \label{prop:center}
    The center of the algebra $\Yt_\mu^\la$ is the polynomial algebra $\bC[\bm z]$. 
\end{prop}

\begin{proof}
Our proof is inspired by \cite[Theorem 4.1(4)]{FO10}.  First note that every element $x \in \Yt_\mu^\la$ can be written uniquely as a sum $x = \sum x_{\bm a} \dfo^{\bm a}$, where we have used multi-index notation $\dfo^{\bm a} = \prod_{i,r} \dfo_{i,r}^{a_{i,r}}$ with all $a_{i,r} \in \bZ$, and where  
$$x_{\bm a} \in \bC[\bm z](w_{i,r} : i \in \iI, 1 \lle r \lle \fkv_i).$$
We'll refer to the $x_{\bm a}$ as the {\em coefficients} of $x$, and will assume from now on that $x$ is central in $\Yt_\mu^\la$.  

We first claim that $x_{\bm a} = 0$ for all non-trivial $\bm a \neq \mathbf{0}$, so that $x = x_{\mathbf{0}} \in \bC[\bm z](w_{i,r})$.
Indeed, suppose that some $x_{\bm a} \neq 0$. By an application of  \cite[Lemma 2.1(4)]{FO10}, there exists an element $f \in \bC[\bm z][\mathsf A_i^{(r)}]$ such that $\dfo^{\mathbf a}f \neq f \dfo^{\mathbf a}$. But then $f x \neq x f$: the coefficient of $\dfo^{\bm a}$ on the left side is $f x_{\bm a}$, while the coefficient on the right side is $(\dfo^{\bm a} f \dfo^{-\bm a}) x_{\bm a} \neq f x_{\bm a}$. 

We next claim that in fact $x \in \bC[\bm z]$. To simplify notation in the proof, let us think of $x = x(w_{i,r})$ as a function of the variables $w_{i,r}$.  For any $j \in \iI$, consider the commutator $[ \Phi_\mu^\la(B_j(u)), x] = 0$.  Up to multiplying by some non-zero rational function, the coefficient of $\dfo_{j,s}$ on the left side is $x(w_{i,r}+ \delta_{i,j}\delta_{r,s}) - x(w_{i,r})$.  This must be zero for all $j,s$, and since $x$ is rational in the variables $w_{i,r}$ the only possibility is that $x$ is constant. In other words, $x \in \bC[\bm z]$, as claimed.
\end{proof}

\begin{rem}
    A similar argument proves that 
    $$
    \Yt_\mu^\la \cap \bC[\bm z](w_{i,r} : i \in \iI, 1 \lle r \lle \fkv_i)
    $$
    is a maximal commutative subalgebra of $\Yt_\mu^\la$, cf.~\cite[Theorem 4.1(3)]{FO10}.
\end{rem}

\section{Proofs of the $\mathrm{i}$GKLO homomorphism theorems}
\label{sec:proof1}

In this section, we shall prove Theorems \ref{thm:iGKLO:gl} and \ref{thm:GKLOquasisplit}.
We prove Theorem \ref{thm:iGKLO:gl} in \S\ref{ssec:bb}--\S\ref{sec:serre}. Then we prove Theorem \ref{thm:GKLOquasisplit} in \S\ref{ssec:qspf} excluding quasi-split type A$_{2n}$ and finish the remaining case in \S\ref{ssec:qsAIII2n}.

It is clear that the image of $D_i(u)$ under $\Phi_{\mu}^\la$ is of the form $u^{- \langle\mu,\ve_i\rangle}+ \text{(lower order terms in }u)$. 
We need to verify that the relations \eqref{ddgl}--\eqref{serregl} are preserved by the map $\Phi_{\mu}^\la$.
The relations \eqref{ddgl}, \eqref{deven}, and \eqref{ee=0} are obvious. 

We shall verify the remaining relations \eqref{de}--\eqref{ee2}, \eqref{ee} and \eqref{serregl}, respectively.

\subsection{The relations \eqref{de}--\eqref{ee2}} 
\label{ssec:bb}

We start with the relation \eqref{de}. To simplify the notation, we write
\begin{align*}
D_i(u)&\mapsto \W_i(u-\tfrac{i-1}{2})\varphi_i(u),\\
E_i(u)&\mapsto -\sum_{r=1}^{\fkv_i}\frac{\xi^+_{i,r}(w_{i,r})}{u-\tfrac{i-1}{2}-w_{i,r}}\dfo_{i,r}^{-1}-\sum_{r=1}^{\fkv_i}\frac{\xi^-_{i,r}(w_{i,r})}{u-\tfrac{i-1}{2}+w_{i,r}}\dfo_{i,r}+\frac{\xi_{i}^0}{u-\tfrac{i}{2}}.
\end{align*}
The functions $\varphi_i(u)$, $\xi_{i,r}^\pm(w_{i,r})$, and $\xi_{i}^0$ can be read off directly from formulas in Theorem \ref{thm:iGKLO:gl} and we do not need their explicit forms. 
Set
\beq\label{diamond}
\W_{i,r}^\diamond(u) =u^{\theta_i} \prod_{\stackrel{s=1}{s\ne r}}^{\fkv_i}(u^2-w_{i,s}^2).
\eeq

Clearly, the image of $[D_i(u),E_i(v)]$ under $\Phi_{\mu}^\la$ is given by
\begin{align*}
-&\sum_{r=1}^{\fkv_i}\frac{\W_{i,r}^\diamond(u-\tfrac{i-1}{2})\varphi_i(u)\xi^+_{i,r}(w_{i,r})}{v-\tfrac{i-1}{2}-w_{i,r}}\big((u-\tfrac{i-1}{2})^2-w_{i,r}^2\big)\dfo_{i,r}^{-1}\\
-&\sum_{r=1}^{\fkv_i}\frac{\W_{i,r}^\diamond(u-\tfrac{i-1}{2})\varphi_i(u)\xi^-_{i,r}(w_{i,r})}{v-\tfrac{i-1}{2}+w_{i,r}}\big((u-\tfrac{i-1}{2})^2-w_{i,r}^2\big)\dfo_{i,r}\\
+&\sum_{r=1}^{\fkv_i}\frac{\W_{i,r}^\diamond(u-\tfrac{i-1}{2})\varphi_i(u)\xi^+_{i,r}(w_{i,r})}{v-\tfrac{i-1}{2}-w_{i,r}}\big((u-\tfrac{i-1}{2})^2-(w_{i,r}-1)^2\big)\dfo_{i,r}^{-1}\\
+&\sum_{r=1}^{\fkv_i}\frac{\W_{i,r}^\diamond(u-\tfrac{i-1}{2})\varphi_i(u)\xi^-_{i,r}(w_{i,r})}{v-\tfrac{i-1}{2}+w_{i,r}}\big((u-\tfrac{i-1}{2})^2-(w_{i,r}+1)^2\big)\dfo_{i,r}\\
=\;\;&\sum_{r=1}^{\fkv_i}\frac{\W_{i,r}^\diamond(u-\tfrac{i-1}{2})\varphi_i(u)\xi^+_{i,r}(w_{i,r})}{v-\tfrac{i-1}{2}-w_{i,r}}\big(2w_{i,r}-1\big)\dfo_{i,r}^{-1}\\
&\hskip2cm -\sum_{r=1}^{\fkv_i}\frac{\W_{i,r}^\diamond(u-\tfrac{i-1}{2})\varphi_i(u)\xi^-_{i,r}(w_{i,r})}{v-\tfrac{i-1}{2}+w_{i,r}}\big(1+2w_{i,r}\big)\dfo_{i,r}.
\end{align*}
On the other hand, the image of $\frac{1}{u-v}D_i(u)(E_i(u)-E_i(v))$ under $\Phi_\mu^\la$ simplifies as
\begin{align*}
&\sum_{r=1}^{\fkv_i}\frac{\W_{i,r}^\diamond(u-\tfrac{i-1}{2})\varphi_i(u)\xi_{i,r}^+(w_{i,r})}{v-\tfrac{i-1}{2}-w_{i,r}}\big(u-\tfrac{i-1}{2}+w_{i,r}\big)\dfo_{i,r}^{-1}\\
+&\sum_{r=1}^{\fkv_i}\frac{\W_{i,r}^\diamond(u-\tfrac{i-1}{2})\varphi_i(u)\xi_{i,r}^-(w_{i,r})}{v-\tfrac{i-1}{2}+w_{i,r}}\big(u-\tfrac{i-1}{2}-w_{i,r}\big)\dfo_{i,r}-\frac{\W_i(u-\tfrac{i-1}{2})\varphi_i(u)\xi_i^0}{(u-\tfrac{i}{2})(v-\tfrac{i}{2})}
\end{align*}
while the image of $\frac{1}{u+v-i}(E_i(v)-E_i(-u+i))D_i(u)$ under $\Phi_\mu^\la$ simplifies as
\begin{align*}
-&\sum_{r=1}^{a_i}\frac{\W_{i,r}^\diamond(u-\tfrac{i-1}{2})\varphi_i(u)\xi_{i,r}^+(w_{i,r})}{v-\tfrac{i-1}{2}-w_{i,r}}\big(u-\tfrac{i-1}{2}-(w_{i,r}-1)\big)\dfo_{i,r}^{-1}\\
-&\sum_{r=1}^{a_i}\frac{\W_{i,r}^\diamond(u-\tfrac{i-1}{2})\varphi_i(u)\xi_{i,r}^-(w_{i,r})}{v-\tfrac{i-1}{2}+w_{i,r}}\big(u-\tfrac{i-1}{2}+(w_{i,r}+1)\big)\dfo_{i,r}+\frac{\W_i(u-\tfrac{i-1}{2})\varphi_i(u)\xi_i^0}{(u-\tfrac{i}{2})(v-\tfrac{i}{2})}.
\end{align*}
Summing up the above formulas, we see that the image of the right-hand side of \eqref{de} coincides with the image of the left-hand side, proving the relation \eqref{de}.

The proof of the relation \eqref{ed} is very similar, which eventually reduces to the following identities:
\begin{align*}
&\frac{1}{(u-\tfrac{i+1}2)^2-w_{i,r}^2}-\frac{1}{(u-\tfrac{i+1}2)^2-(w_{i,r}\pm 1)^2}\\
=&\frac{1}{\big((u-\tfrac{i+1}2)^2-w_{i,r}^2\big)\big(u-\tfrac{i-1}2\pm w_{i,r}\big)}-\frac{1}{\big((u-\tfrac{i+1}2)^2-(w_{i,r}-1)^2\big)\big(u-\tfrac{i+1}2\mp w_{i,r}\big)}.
\end{align*}

To simplify the notation in the next task, following \cite{GKLO} we set $w_{i,0}=\tfrac12$ and introduce
\begin{align}
\chi_{i,0}^+&= \sqrt{(-1)^{\fkw_i+\fkv_{i-1}+\fkv_{i+1}}}\frac{\theta_i Z_i(0)}{\W^\circ_{i}(w_{i,0})}\prod_{j\leftrightarrow i}W_j(0),\label{chi0def}\\
\chi_{i,r}^+&=-\frac{\varkappa(w_{i,r})^{-\vartheta_i}\W_{i-1}(w_{i,r}-\tfrac12)\overline W^-_{i+1}(w_{i,r}-\tfrac12)}{\W_{i,r}(w_{i,r})}\dfo_{i,r}^{-1},\notag\\
\chi_{i,r}^-&=-\frac{\varkappa(-w_{i,r})^{-\vartheta_i}W_{i+1}(w_{i,r}+\tfrac12)\Z_i(w_{i,r}+\tfrac12)}{\W_{i,r}(w_{i,r})}\dfo_{i,r},
\qquad \text{ for } 1\lle r\lle \fkv_i.
\notag
\end{align}
Note that $\chi_{i,0}^+=0$ when $\theta_i=0$. 
Then from \eqref{Bi:shiftEi} and \eqref{eq:Eiu} we derive that
\begin{align}  \label{Bi:chi}
\Phi_{\mu}^\la(B_i(u)) &=\sum_{r=0}^{\fkv_i}\frac{1}{u+\tfrac{1}{2}-w_{i,r}}\chi_{i,r}^++\sum_{r=1}^{\fkv_i}\frac{1}{u+\tfrac{1}{2}+w_{i,r}}\chi_{i,r}^-,
\\
\Phi_{\mu}^\la(B_i^{(1)}) &=\sum_{r=0}^{\fkv_i}\chi_{i,r}^++\sum_{r=1}^{\fkv_i}\chi_{i,r}^-.
\notag
\end{align}

Now we proceed to prove the relation \eqref{ee2}, and we shall set $j=i+1$. Note that $\theta_i\theta_j=0$ by \eqref{parity}. Without loss of generality, we assume that $\theta_j=0$. 

\begin{lem}\label{lem:chi}
Let $j=i+1$. We have
\begin{align*}
(\pm w_{i,r}\mp w_{i,s}-1)\chi_{i,r}^\pm\chi_{i,s}^\pm&=(\pm w_{i,r}\mp w_{i,s}+1)\chi_{i,s}^\pm\chi_{i,r}^\pm,\qquad \text{if }r\ne s,\\
(\pm w_{i,r}\pm w_{i,s}-1)\chi_{i,r}^\pm\chi_{i,s}^\mp&=(\pm w_{i,r}\pm w_{i,s}+1)\chi_{i,s}^\mp\chi_{i,r}^\pm,\qquad \text{if }r\ne s,\\
(\pm w_{i,r}\mp  w_{j,s}+\tfrac12)\chi_{i,r}^\pm\chi_{j,s}^\pm&=(\pm w_{i,r}\mp w_{j,s}-\tfrac12)\chi_{j,s}^\pm\chi_{i,r}^\pm,\\
(\pm w_{i,r}\pm  w_{j,s}+\tfrac12)\chi_{i,r}^\pm\chi_{j,s}^\mp&=(\pm w_{i,r}\pm w_{j,s}-\tfrac12)\chi_{j,s}^\mp\chi_{i,r}^\pm,
\end{align*}
where $r=0$ is allowed in $\chi_{i,r}^+$.
\end{lem}
\begin{proof}
Follows from a direct computation.
\end{proof}

\begin{rem}
Alternatively, one can also set $w_{i,0}=-\tfrac12$ and introduce
\be
\chi_{i,0}^-= \sqrt{(-1)^{\fkw_i+\fkv_{i-1}+\fkv_{i+1}}}\frac{\theta_i Z_i(0)}{\W^\circ_{i}(w_{i,0})}\prod_{j\leftrightarrow i}W_j(0).
\ee
Then the above lemma also holds for $\chi_{i,r}^-$ with $r=0$.
\end{rem}

The relation \eqref{ee2} can be equivalently written as
\beq\label{eepf3}
\begin{split}
&(u-v+\tfrac12)B_i(u)B_j(v)+B_i(u)B_j^{(1)}-B_i^{(1)}B_j(v)\\
=&(u-v-\tfrac12)B_j(v)B_i(u)+B_j^{(1)}B_i(u)-B_j(v)B_i^{(1)}.
\end{split}
\eeq
The image of the left-hand side of \eqref{eepf3} under $\Phi_\mu^\la$ simplifies as
\begin{align*}
\sum_{r=0}^{\fkv_i}\sum_{s=1}^{\fkv_j}\bigg(&\frac{w_{i,r}-w_{j,s}+\tfrac12}{(u+\tfrac12-w_{i,r})(v+\tfrac12-w_{j,s})}\chi_{i,r}^+\chi_{j,s}^+ +  \frac{w_{i,r}+w_{j,s}+\tfrac12}{(u+\tfrac12-w_{i,r})(v+\tfrac12+w_{j,s})}\chi_{i,r}^+\chi_{j,s}^-\bigg)\\
+\sum_{r=1}^{\fkv_i}\sum_{s=1}^{\fkv_j}\bigg(&\frac{-w_{i,r}-w_{j,s}+\tfrac12}{(u+\tfrac12+w_{i,r})(v+\tfrac12-w_{j,s})}\chi_{i,r}^-\chi_{j,s}^+ +  \frac{-w_{i,r}+w_{j,s}+\tfrac12}{(u+\tfrac12+w_{i,r})(v+\tfrac12+w_{j,s})}\chi_{i,r}^-\chi_{j,s}^-\bigg)
\end{align*}
while the image of the left-hand side of \eqref{eepf3} under $\Phi_\mu^\la$ is
\begin{align*}
\sum_{r=0}^{\fkv_i}\sum_{s=1}^{\fkv_j}\bigg(&\frac{w_{i,r}-w_{j,s}-\tfrac12}{(u+\tfrac12-w_{i,r})(v+\tfrac12-w_{j,s})}\chi_{j,s}^+\chi_{i,r}^+ +  \frac{w_{i,r}+w_{j,s}-\tfrac12}{(u+\tfrac12-w_{i,r})(v+\tfrac12+w_{j,s})}\chi_{j,s}^-\chi_{i,r}^+\bigg)\\
+\sum_{r=1}^{\fkv_i}\sum_{s=1}^{\fkv_j}\bigg(&\frac{-w_{i,r}-w_{j,s}-\tfrac12}{(u+\tfrac12+w_{i,r})(v+\tfrac12-w_{j,s})}\chi_{j,s}^+\chi_{i,r}^- +  \frac{-w_{i,r}+w_{j,s}-\tfrac12}{(u+\tfrac12+w_{i,r})(v+\tfrac12+w_{j,s})}\chi_{j,s}^-\chi_{i,r}^-\bigg).
\end{align*}
Clearly, those two images coincide by Lemma \ref{lem:chi}, completing the verification of the relation \eqref{ee2}.

\subsection{The relation \eqref{ee}}\label{ssec:eepf}
From now on, we set $\tl u_i=u-\tfrac{i-1}2$ and $\tl v_i=v-\tfrac{i-1}2$. Then we have
\begin{align*}
E_i(u)\mapsto \sum_{r=0}^{\fkv_i}\frac{1}{\tl u_i-w_{i,r}}\chi_{i,r}^++\sum_{s=1}^{\fkv_i}\frac{1}{\tl u_i+w_{i,s}}\chi_{i,s}^-.
\end{align*}
For simplicity, we set
\beq\label{eq:Omega}
\Omega_i(u):=\frac{\bm\varkappa(u-\tfrac12)^{\vartheta_i}\W_{i-1}(u-\tfrac{1}{2})\W_{i+1}(u-\tfrac{1}{2})\Z_i(u-\tfrac{1}2)}{\W_i(u)\W_i(u-1)}.
\eeq
Here $\Omega_i(u)=\Omega_i(-u+1)$. Moreover, if $\theta_i=\vartheta_i$, then
\beq\label{eq:Omega2}
\Omega_i(u)=\frac{\W_{i-1}(u-\tfrac{1}{2})\W_{i+1}(u-\tfrac{1}{2})\Z_i(u-\tfrac{1}2)}{(u-\hf)^{2\theta_i}\W_i^\circ(u)\W_i^\circ(u-1)}.
\eeq
To verify the relation \eqref{ee}, it is equivalent to check that 
\beq\label{eerel}
\begin{split}
&\bigg[\sum_{r=0}^{\fkv_i}\frac{1}{\tl u_i-w_{i,r}}\chi_{i,r}^++\sum_{r=1}^{\fkv_i}\frac{1}{\tl u_i+w_{i,r}}\chi_{i,r}^-,\sum_{s=0}^{\fkv_i}\frac{1}{\tl v_i-w_{i,s}}\chi_{i,s}^++\sum_{s=1}^{\fkv_i}\frac{1}{\tl v_i+w_{i,s}}\chi_{i,s}^-\bigg]\\
=&-(\tl u_i-\tl v_i)\bigg(\sum_{r=0}^{\fkv_i}\frac{1}{(\tl u_i-w_{i,r})(\tl v_i-w_{i,r})}\chi_{i,r}^+ + \sum_{s=1}^{\fkv_i}\frac{1}{(\tl u_i+w_{i,s})(\tl v_i+w_{i,s})}\chi_{i,s}^-\bigg)^2\\
&\hskip6.4cm -\frac{1}{\tl u_i+\tl v_i-1}\Big(\big(\Omega_i(\tl u_i)\big)^\star-\big(\Omega_i(\tl v_i)\big)^\star\Big).
\end{split}
\eeq
We shall move all difference operators $\chi_{i,r}^{\pm}$ to the right and then compare the terms not involving difference operators and terms containing $\chi_{i,r}^{\pm}\chi_{i,s}^{\pm}$ and  $\chi_{i,r}^{\pm}\chi_{i,s}^{\mp}$.

\subsubsection{Terms involving $\chi_{i,r}^{\pm}\chi_{i,s}^{\pm}$ and  $\chi_{i,r}^{\pm}\chi_{i,s}^{\mp}$ $(r\ne s)$}
Fix $0\lle r,s\lle \fkv_i$ such that $r\ne s$. We collect terms containing $\chi_{i,r}^{+}\chi_{i,s}^{+}$ and $\chi_{i,r}^{+}\chi_{i,s}^{+}$. The terms containing $\chi_{i,r}^{+}\chi_{i,s}^{+}$ and $\chi_{i,s}^{+}\chi_{i,r}^{+}$ from the left-hand side of \eqref{eerel} are given by
\beq\label{ddl}
\begin{split}
&\frac{1}{(\tl u_i-w_{i,r})(\tl v_i-w_{i,s})}\chi_{i,r}^+\chi_{i,s}^+-\frac{1}{(\tl u_i-w_{i,r})(\tl v_i-w_{i,s})}\chi_{i,s}^+\chi_{i,r}^+\\
&\qquad \qquad +\frac{1}{(\tl u_i-w_{i,s})(\tl v_i-w_{i,r})}\chi_{i,s}^+\chi_{i,r}^+-\frac{1}{(\tl u_i-w_{i,s})(\tl v_i-w_{i,r})}\chi_{i,r}^+\chi_{i,s}^+\\
&= \frac{-(\tl u_i-\tl v_i)\big((w_{i,r}-w_{i,s})\chi_{i,r}^+\chi_{i,s}^++(w_{i,s}-w_{i,r})\chi_{i,s}^+\chi_{i,r}^+\big)}{(\tl u_i-w_{i,r})(\tl u_i-w_{i,s})(\tl v_i-w_{i,r})(\tl v_i-w_{i,s})}.
\end{split}
\eeq
The terms containing $\chi_{i,r}^{+}\chi_{i,s}^{+}$ and $\chi_{i,s}^{+}\chi_{i,r}^{+}$ from the right-hand side of \eqref{eerel} are given by
\beq\label{ddr}
\begin{split}
\frac{-(\tl u_i-\tl v_i)\big( \chi_{i,r}^+\chi_{i,s}^++ \chi_{i,s}^+\chi_{i,r}^+\big)}{(\tl u_i-w_{i,r})(\tl u_i-w_{i,s})(\tl v_i-w_{i,r})(\tl v_i-w_{i,s})}.
\end{split}
\eeq
To prove that \eqref{ddl}\,=\,\eqref{ddr}, it is equivalent  to prove that
\begin{align}
\frac{-(\tl u_i-\tl v_i)\big((w_{i,r}-w_{i,s}-1)\chi_{i,r}^+\chi_{i,s}^++(w_{i,s}-w_{i,r}-1)\chi_{i,s}^+\chi_{i,r}^+\big)}{(\tl u_i-w_{i,r})(\tl u_i-w_{i,s})(\tl v_i-w_{i,r})(\tl v_i-w_{i,s})}=0,\label{ddpf1}
\end{align}
which follows immediately from Lemma \ref{lem:chi}.

To prove that the terms containing $\chi_{i,r}^{\pm}\chi_{i,s}^{\pm}$ and $\chi_{i,s}^{\pm}\chi_{i,r}^{\pm}$ (resp. $\chi_{i,r}^{\pm}\chi_{i,s}^{\mp}$ and $\chi_{i,s}^{\mp}\chi_{i,r}^{\pm}$) on both sides match, the computation is almost identical. Indeed, it essentially reduces to prove the identity \eqref{ddpf1} with $w_{i,r}$ and $w_{i,s}$ replaced by $\pm w_{i,r}$ and $\pm w_{i,s}$ (resp. $\pm w_{i,r}$ and $\mp w_{i,s}$), respectively. Note that we allow $r=0$ in $\chi_{i,r}^+$.

\subsubsection{Terms involving $\chi_{i,r}^{\pm}\chi_{i,r}^{\pm}$ $(r\gge 1)$ }
The terms containing $\chi_{i,r}^{+}\chi_{i,r}^{+}$ from the left-hand side of \eqref{eerel} are given by
\beq\label{ddl2}
\begin{split}
&\frac{1}{\tl u_i-w_{i,r}}\chi_{i,r}^+\frac{1}{\tl v_i-w_{i,r}}\chi_{i,r}^+-\frac{1}{\tl v_i-w_{i,r}}\chi_{i,r}^+\frac{1}{\tl u_i-w_{i,r}}\chi_{i,r}^+\\
=&\Big(\frac{1}{(\tl u_i-w_{i,r})(\tl v_i-w_{i,r}+1)}-\frac{1}{(\tl v_i-w_{i,r})(\tl u_i-w_{i,r}+1)}\Big)\chi_{i,r}^+\chi_{i,r}^+\\
=&\frac{-(\tl u_i-\tl v_i)}{(\tl u_i-w_{i,r})(\tl v_i-w_{i,r})(\tl u_i-w_{i,r}+1)(\tl v_i-w_{i,r}+1)}\chi_{i,r}^+\chi_{i,r}^+
\end{split}
\eeq
while the terms containing $\chi_{i,r}^{+}\chi_{i,r}^{+}$ from the right-hand side of \eqref{eerel} are given by 
\beq\label{ddr2}
-(\tl u_i-\tl v_i)\frac{1}{(\tl u_i-w_{i,r})(\tl v_i-w_{i,r})}\chi_{i,r}^+\frac{1}{(\tl u_i-w_{i,r})(\tl v_i-w_{i,r})}\chi_{i,r}^+.
\eeq
It is clear that \eqref{ddl2}\,=\,\eqref{ddr2} since $\chi_{i,r}^+$ contains $\dfo_{i,r}^{-1}$. 

The case for terms of $\chi_{i,r}^{-}\chi_{i,r}^{-}$ is reduced to a similar computation with $w_{i,r}$ replaced by $-w_{i,r}$.

\subsubsection{The constant terms}
We are left with comparing the terms not involving difference operators. We call such terms \emph{constant terms}. The constant terms from the left-hand side of \eqref{eerel} are given by
\begin{align*}
&\sum_{r=1}^{\fkv_i}\frac{1}{(\tl u_i-w_{i,r})(\tl v_i+w_{i,r}-1)}\chi_{i,r}^+\chi_{i,r}^- + \sum_{r=1}^{\fkv_i}\frac{1}{(\tl u_i+w_{i,r})(\tl v_i-w_{i,r}-1)}\chi_{i,r}^-\chi_{i,r}^+ \\
-&\sum_{r=1}^{\fkv_i}\frac{1}{(\tl v_i-w_{i,r})(\tl u_i+w_{i,r}-1)}\chi_{i,r}^+\chi_{i,r}^- - \sum_{r=1}^{\fkv_i}\frac{1}{(\tl u_i-w_{i,r}-1)(\tl v_i+w_{i,r})}\chi_{i,r}^-\chi_{i,r}^+.
\end{align*}
The constant terms from the right-hand side of \eqref{eerel} are
\begin{align*}
-(\tl u_i-\tl v_i)\bigg(&
\sum_{r=1}^{\fkv_i}\frac{1}{(\tl u_i-w_{i,r})(\tl v_i-w_{i,r})(\tl u_i+w_{i,r}-1)(\tl v_i+w_{i,r}-1)}\chi_{i,r}^+\chi_{i,r}^-&\\
&+\sum_{r=1}^{\fkv_i}\frac{1}{(\tl u_i+w_{i,r})(\tl v_i+w_{i,r})(\tl u_i-w_{i,r}-1)(\tl v_i-w_{i,r}-1)}\chi_{i,r}^-\chi_{i,r}^+
&\\
&+\frac{1}{(\tl u_i-\tfrac12)^2(\tl v_i-\tfrac12)^2}\chi_{i,0}^+\chi_{i,0}^+ \bigg)-\frac{1}{\tl u_i+\tl v_i-1}\Big(\big(\Omega_i(\tl u_i)\big)^\star-\big(\Omega_i(\tl v_i)\big)^\star\Big),
\end{align*}
which can be further simplified as
\begin{align*}
&\sum_{r=1}^{\fkv_i}\frac{1}{\brown{(2w_{i,r}-1)}}\Big(\frac{1}{(\tl u_i-w_{i,r})(\tl v_i+w_{i,r}-1)}\chi_{i,r}^+\chi_{i,r}^- -\frac{1}{(\tl u_i+w_{i,r}-1)(\tl v_i-w_{i,r})}\chi_{i,r}^+\chi_{i,r}^- \Big)\\
+&\sum_{r=1}^{\fkv_i}\frac{1}{\brown{(2w_{i,r}+1)}}\Big(\frac{1}{(\tl u_i-w_{i,r}-1)(\tl v_i+w_{i,r})}\chi_{i,r}^-\chi_{i,r}^+ -\frac{1}{(\tl u_i+w_{i,r})(\tl v_i-w_{i,r}-1)}\chi_{i,r}^-\chi_{i,r}^+ \Big)\\
&\hskip1.45cm-\frac{1}{\tl u_i+\tl v_i-1}\Big(\big(\Omega_i(\tl u_i)\big)^\star-\big(\Omega_i(\tl v_i)\big)^\star-\frac{1}{(\tl u_i-\tfrac12)^2}\chi_{i,0}^+\chi_{i,0}^++\frac{1}{(\tl v_i-\tfrac12)^2}\chi_{i,0}^+\chi_{i,0}^+\Big).
\end{align*}
Hence, we need to prove that
\begin{align*}
&-\frac{1}{\tl u_i+\tl v_i-1}\Big(\big(\Omega_i(\tl u_i)\big)^\star-\big(\Omega_i(\tl v_i)\big)^\star-\frac{1}{(\tl u_i-\tfrac12)^2}\chi_{i,0}^+\chi_{i,0}^++\frac{1}{(\tl v_i-\tfrac12)^2}\chi_{i,0}^+\chi_{i,0}^+\Big)\\
=&\sum_{r=1}^{\fkv_i}\frac{\brown{2(w_{i,r}-1)}}{\brown{(2w_{i,r}-1)}}\Big(\frac{1}{(\tl u_i-w_{i,r})(\tl v_i+w_{i,r}-1)}\chi_{i,r}^+\chi_{i,r}^- -\frac{1}{(\tl u_i+w_{i,r}-1)(\tl v_i-w_{i,r})}\chi_{i,r}^+\chi_{i,r}^- \Big)\\
+&\sum_{r=1}^{\fkv_i}\frac{\brown{2(w_{i,r}+1)}}{\brown{(2w_{i,r}+1)}}\Big(\frac{1}{(\tl u_i+w_{i,r})(\tl v_i-w_{i,r}-1)}\chi_{i,r}^-\chi_{i,r}^+ - \frac{1}{(\tl u_i-w_{i,r}-1)(\tl v_i+w_{i,r})}\chi_{i,r}^-\chi_{i,r}^+\Big).
\end{align*}
Set
\[
\W^\circ_{i,r}=(u+w_{i,r})\prod_{1\lle s\lle \fkv_i}^{s\ne r}(u^2-w_{i,s}^2).
\]
Note that, for $r\gge 1$,
\[
\dfo_{i,r}^{\pm 1}\W_{i,r}^\circ(w_{i,r})=\W_{i,r}^\circ(w_{i,r}\pm 1)\brown{\frac{2(w_{i,r}\pm 1)}{2w_{i,r}\pm 1}}\dfo_{i,r}^{\pm 1},
\]
and thus we have 
\begin{align}
\chi_{i,r}^+\chi_{i,r}^-&=-\frac{\brown{(2w_{i,r}-1)}}{\brown{2(w_{i,r}-1)}}\mathrm{Res}_{u=w_{i,r}}\Omega_i(u)=\frac{\brown{(2w_{i,r}-1)}}{\brown{2(w_{i,r}-1)}}\mathrm{Res}_{u=-w_{i,r}+1}\Omega_i(u),\label{chi+-}\\
\chi_{i,r}^-\chi_{i,r}^+&=-\frac{\brown{(2w_{i,r}+1)}}{\brown{2(w_{i,r}+1)}}\mathrm{Res}_{u=-w_{i,r}}\Omega_i(u)=\frac{ \brown{(2w_{i,r}+1)}}{\brown{2(w_{i,r}+1)}}\mathrm{Res}_{u=w_{i,r}+1}\Omega_i(u),\label{chi-+}
\end{align}
where $\mathrm{Res}_{u=a}f(u)$ denotes the residue of $f(u)$ at $u=a$.

Note that by \eqref{chi0def}, we have
\beq\label{chisqu}
\chi_{i,0}^+\chi_{i,0}^+=\frac{\theta_i\W_{i-1}(0)\W_{i+1}(0)\Z_i(0)}{\W_i^\circ(\tfrac12)\W_i^\circ(\tfrac12)}.
\eeq
Hence, to show that the constant terms from both sides of \eqref{eerel} match, it suffices to prove that
\begin{align*}
\big(\Omega_i(\tl v_i)\big)^\star-\big(\Omega_i(\tl u_i)\big)^\star
\end{align*}
is equal to
\begin{align*}
&\theta_i\bigg(\frac{1}{(\tl v_i-\tfrac12)^2}-\frac{1}{(\tl u_i-\tfrac12)^2}\bigg)\frac{\W_{i-1}(0)\W_{i+1}(0)\Z_i(0)}{\W_i^\circ(\tfrac12)\W_i^\circ(-\tfrac12)}\\
+\,&\sum_{r=1}^{\fkv_i}\mathrm{Res}_{u=-w_{i,r}+1}\Omega_i(u)\bigg(\frac{1}{\tl u_i-w_{i,r}}
+\frac{1}{\tl v_i+w_{i,r}-1}-\frac{1}{\tl u_i+w_{i,r}-1}-\frac{1}{\tl v_i-w_{i,r}}\bigg)\\
+\,&\sum_{r=1}^{\fkv_i}\mathrm{Res}_{u=w_{i,r}+1}\Omega_i(u)\bigg(\frac{1}{\tl u_i+w_{i,r}}
+\frac{1}{\tl v_i-w_{i,r}-1}-\frac{1}{\tl u_i-w_{i,r}-1}-\frac{1}{\tl v_i+w_{i,r}}\bigg);
\end{align*}
this follows from the following standard result, cf. \cite[proof of Theorem 4.5]{KWWY14} and \cite[Appendix~ B.5.2]{BFN19}.

\begin{lem}\label{lem:std}
\begin{enumerate}
    \item For any rational function $\gamma(u)$ with simple poles $\{a_k\}\subset \bC$ and a possible pole of higher order at $u=\infty$, we have
\[
\gamma(u)^\star=\sum_{k}\frac{1}{u-a_k}\mathrm{Res}_{u=a_k}\gamma(u).
\]
\item Let $\theta_i\in\{0,1\}$ be as before. For any even rational function $\gamma(u)$ with simple poles $\{a_k\}\subset \bC^\times$ and a possible pole of higher order at $u=\infty$, we have
\[
\big(u^{-2\theta_i}\gamma(u)\big)^\star=\frac{1}{u^2}\theta_i\gamma(0)+\sum_{k}\frac{1}{u-a_k}\mathrm{Res}_{u=a_k}u^{-2\theta_i}\gamma(u).
\]
\end{enumerate}
\end{lem}

Indeed, we have 2 cases. If $\theta_i=\vartheta_i$, then the desired equality follows from Lemma \ref{lem:std} and \eqref{eq:Omega2}. If $\theta_i\ne\vartheta_i$, then it only happens if $\theta_i=0$ and $\vartheta_i=1$. Thus $\max\{\theta_{i-1},\theta_{i+1}\}=1$ and at least one of $\W_{i-1}(u)$ and $\W_{i+1}(u)$ are odd. It follows that $\Omega_i(u)$ has at most a simple pole at $u=\hf$. However, $\Omega_i(u+\hf)$ is even and hence $\Omega_i(u)$ must be regular at $u=\hf$. 

\subsection{The Serre relation \eqref{serregl}} \label{sec:serre}

We prove it for the case $j=i+1$ as the other case $j=i-1$ is similar. In the remainder of this section, sometimes we write $\frac{p\chi_{i,r}^+}{q}$ for $\frac{p}{q}\chi_{i,r}^+$ to shorten formulas, where $p,q$ are polynomials in $u$ and $w_{j,s}$, and similarly for others. We understand these ratio as follows: the terms $\chi_{i,r}^\pm$ involving difference operators are always to the right of the scalar rational functions. It is also convenient to set $\chi_{i,0}^-=\chi_{j,0}^-=0$.

We start with the left-hand side of \eqref{serregl}. We have
\begin{align}
&\big[\Phi_{\mu}^\la(B_i^{(1)}),[\Phi_{\mu}^\la(B_i(u)),\Phi_{\mu}^\la(B_j^{(1)})]\big]+\big[\Phi_{\mu}^\la(B_i(u)),[\Phi_{\mu}^\la(B_i^{(1)}),\Phi_{\mu}^\la(B_j^{(1)})]\big]\label{serpf1}\\
=&\bigg[\sum_{r_1=0}^{\fkv_i}\big(\chi_{i,r_1}^++\chi_{i,r_1}^-\big),\Big[\sum_{r_2=0}^{\fkv_i}\Big(\frac{\chi_{i,r_2}^+}{u+\tfrac{1}{2}-w_{i,r_2}}+\frac{\chi_{i,r_2}^-}{u+\tfrac{1}{2}+w_{i,r_2}}\Big),\sum_{s=0}^{\fkv_j}\big(\chi_{j,s}^++\chi_{j,s}^-\big)\Big]\bigg]\notag\\
+&\bigg[\sum_{r_1=0}^{\fkv_i}\Big(\frac{\chi_{i,r_1}^+}{u+\tfrac{1}{2}-w_{i,r_1}}+\frac{\chi_{i,r_2}^-}{u+\tfrac{1}{2}+w_{i,r_1}}\Big),\Big[\sum_{r_2=0}^{\fkv_i}\big(\chi_{i,r_2}^++\chi_{i,r_2}^-\big),\sum_{s=0}^{\fkv_j}\big(\chi_{j,s}^++\chi_{j,s}^-\big)\Big]\bigg].\notag
\end{align}
We first consider the case $r_1\ne r_2$. By Lemma \ref{lem:chi}, we have
\begin{align*}
&\Big[\frac{1}{u+\tfrac{1}{2}-w_{i,r_1}}\chi_{i,r_1}^+,\big[\chi_{i,r_2}^+ , \chi_{j,s}^+ \big]\Big]+\Big[\chi_{i,r_2}^+,\Big[ \frac{1}{u+\tfrac{1}{2}-w_{i,r_1}}\chi_{i,r_1}^+, \chi_{j,s}^+ \Big]\Big]\\
=&\Big[\frac{1}{u+\tfrac{1}{2}-w_{i,r_1}}\chi_{i,r_1}^+,\frac{-1}{w_{i,r_2}-w_{j,s}-\tfrac12}\chi_{i,r_2}^+  \chi_{j,s}^+\Big]\\
&+\Big[\chi_{i,r_2}^+,\frac{-1}{(u+\tfrac12-w_{i,r_1})(w_{i,r_1}-w_{j,s}-\tfrac12)}\chi_{i,r_1}^+  \chi_{j,s}^+\Big]\\
=&\frac{-1}{(u+\tfrac{1}{2}-w_{i,r_1})(w_{i,r_2}-w_{j,s}-\tfrac12)}\Big(\chi_{i,r_1}^+\chi_{i,r_2}^+\chi_{j,s}^+-\frac{w_{i,r_1}-w_{j,s}+\tfrac12}{w_{i,r_1}-w_{j,s}-\tfrac12}\chi_{i,r_2}^+\chi_{i,r_1}^+\chi_{j,s}^+\Big)\\
&+\frac{-1}{(u+\tfrac{1}{2}-w_{i,r_1})(w_{i,r_1}-w_{j,s}-\tfrac12)}\Big(\chi_{i,r_2}^+\chi_{i,r_1}^+\chi_{j,s}^+-\frac{w_{i,r_2}-w_{j,s}+\tfrac12}{w_{i,r_2}-w_{j,s}-\tfrac12}\chi_{i,r_1}^+\chi_{i,r_2}^+\chi_{j,s}^+\Big)\\
=&\frac{-1}{(u+\tfrac{1}{2}-w_{i,r_1})(w_{i,r_1}-w_{j,s}-\tfrac12)(w_{i,r_2}-w_{j,s}-\tfrac12)}\Big((w_{i,r_1}-w_{i,r_2}-1)\chi_{i,r_1}^+\chi_{i,r_2}^+\chi_{j,s}^+\\
&\hskip8cm +(w_{i,r_2}-w_{i,r_1}-1)\chi_{i,r_2}^+\chi_{i,r_1}^+\chi_{j,s}^+\Big)=0.
\end{align*}
The same type of computation implies that
\[
\Big[\frac{1}{u+\tfrac{1}{2}-w_{i,r_1}}\chi_{i,r_1}^{*_1},\big[\chi_{i,r_2}^{*_2} , \chi_{j,s}^{*_3} \big]\Big]+\Big[\chi_{i,r_2}^{*_2},\Big[ \frac{1}{u+\tfrac{1}{2}-w_{i,r_1}}\chi_{i,r_1}^{*_1}, \chi_{j,s}^{*_3} \Big]\Big]=0
\]
for $*_{i}\in\{\pm\}$, $i=1,2,3$, provided $r_1\ne r_2$. Specifically, one only needs to change $w_{i,r_1}$, $w_{i,r_2}$, $w_{j,s}$ to $*_1w_{i,r_1}$, $*_2w_{i,r_2}$, $*_3w_{j,s}$, respectively, in the above calculation.

Similarly, one finds that
\[
\Big[\frac{1}{u+\tfrac{1}{2}-w_{i,r}}\chi_{i,r}^{*},\big[\chi_{i,r}^{*} , \chi_{j,s}^{\pm} \big]\Big]+\Big[\chi_{i,r}^{*},\Big[ \frac{1}{u+\tfrac{1}{2}-w_{i,r}}\chi_{i,r}^{*}, \chi_{j,s}^{\pm} \Big]\Big]=0
\]
for $*\in\{\pm\}$ for $r\gge 1$. Thus the right-hand side of \eqref{serpf1} is equal to
\beq\label{inv0}
\Big[\frac{1}{u+\tfrac{1}{2}- w_{i,0}}\chi_{i,0}^{+},\big[\chi_{i,0}^{+} , \chi_{j,s}^{*} \big]\Big]+\Big[\chi_{i,0}^{+},\Big[ \frac{1}{u+\tfrac{1}{2}-w_{i,0}}\chi_{i,0}^{+}, \chi_{j,s}^{*} \Big]\Big]
\eeq
plus
\beq\label{invnot0}
\begin{split}
\sum_{*\in\{\pm\}}\sum_{s=0}^{\fkv_j}\sum_{r=1}^{\fkv_i}\bigg(&\Big[\frac{1}{u+\tfrac{1}{2}- w_{i,r}}\chi_{i,r}^{+},\big[\chi_{i,r}^{-} , \chi_{j,s}^{*} \big]\Big]+\Big[\chi_{i,r}^{-},\Big[ \frac{1}{u+\tfrac{1}{2}-w_{i,r}}\chi_{i,r}^{+}, \chi_{j,s}^{*} \Big]\Big]\\
+&\Big[\frac{1}{u+\tfrac{1}{2}+ w_{i,r}}\chi_{i,r}^{-},\big[\chi_{i,r}^{+}, \chi_{j,s}^{*} \big]\Big]+\Big[\chi_{i,r}^{+},\Big[ \frac{1}{u+\tfrac{1}{2}+w_{i,r}}\chi_{i,r}^{-}, \chi_{j,s}^{*} \Big]\Big]\bigg). 
\end{split}
\eeq
We shall prove the Serre relation \eqref{serregl} by compare the terms containing the same $\chi_{j,s}^+$. The case for $\chi_{j,s}^-$ is similar.

Recall our assumption \eqref{parity}, $\chi_{i,0}^+$ and $\chi_{j,0}^+$ cannot be both nonzero. Hence if $s\ne0$, it follows from Lemma \ref{lem:chi} and \eqref{chisqu} that the part \eqref{inv0} can be transformed to
\beq\label{inv0new}
\begin{split}
\frac{2\theta_i\W_{i-1}(0)\W_{j}(0)\Z_i(0)}{(u+\tfrac{1}{2}- w_{i,0})w_{j,s}^2\W_i^\circ(\tfrac12)\W_i^\circ(\tfrac12)}\chi_{j,s}^+.
\end{split}
\eeq

Again using Lemma \ref{lem:chi}, the part \eqref{invnot0} can be rewritten as
\begin{align*}
&\sum_{r=1}^{\fkv_i}\bigg(\Big[\frac{1}{u+\tfrac{1}{2}- w_{i,r}}\chi_{i,r}^{+},\big[\chi_{i,r}^{-} , \chi_{j,s}^{+} \big]\Big]+\Big[\chi_{i,r}^{-},\Big[ \frac{1}{u+\tfrac{1}{2}-w_{i,r}}\chi_{i,r}^{+}, \chi_{j,s}^{+} \Big]\Big]\\
&\hskip0.5cm +\Big[\frac{1}{u+\tfrac{1}{2}+ w_{i,r}}\chi_{i,r}^{-},\big[\chi_{i,r}^{+}, \chi_{j,s}^{+} \big]\Big]+\Big[\chi_{i,r}^{+},\Big[ \frac{1}{u+\tfrac{1}{2}+w_{i,r}}\chi_{i,r}^{-}, \chi_{j,s}^{+} \Big]\Big]\bigg) \\
=& \sum_{r=1}^{\fkv_i}\bigg(\Big[\frac{\chi_{i,r}^{+}}{u+\tfrac{1}{2}- w_{i,r}},\frac{\chi_{i,r}^{-}\chi_{j,s}^{+}}{w_{i,r}+w_{j,s}+\tfrac12} \Big]+\Big[\chi_{i,r}^{-},\frac{-\chi_{i,r}^{+}\chi_{j,s}^{+}}{(u+\tfrac{1}{2}-w_{i,r})(w_{i,r}-w_{j,s}-\tfrac12)} \Big]\\
&\hskip0.5cm +\Big[\frac{\chi_{i,r}^{-}}{u+\tfrac{1}{2}+ w_{i,r}},\frac{-\chi_{i,r}^{+}\chi_{j,s}^{+}}{w_{i,r}-w_{j,s}-\tfrac12} \Big]+\Big[\chi_{i,r}^{+},\frac{\chi_{i,r}^{-}\chi_{j,s}^{+}}{(u+\tfrac{1}{2}+w_{i,r})(w_{i,r}+w_{j,s}+\tfrac12)} \Big]\bigg)\\
=&\sum_{r=1}^{\fkv_i}\bigg(\frac{\chi_{i,r}^+\chi_{i,r}^-\chi_{j,s}^+}{(u+\tfrac{1}{2}- w_{i,r})(w_{i,r}+w_{j,s}-\tfrac12)}-\frac{(w_{i,r}-w_{j,s}+\tfrac32)\chi_{i,r}^-\chi_{i,r}^+\chi_{j,s}^+}{(u-\tfrac{1}{2}- w_{i,r})((w_{i,r}+\tfrac12)^2-w_{j,s}^2)}\\
&\hskip0.5cm -\frac{\chi_{i,r}^-\chi_{i,r}^+\chi_{j,s}^+}{(u-\tfrac{1}{2}- w_{i,r})(w_{i,r}-w_{j,s}+\tfrac12)}+\frac{(w_{i,r}+w_{j,s}-\tfrac32)\chi_{i,r}^+\chi_{i,r}^-\chi_{j,s}^+}{(u+\tfrac{1}{2}- w_{i,r})((w_{i,r}-\tfrac12)^2-w_{j,s}^2)}\\
&\hskip0.5cm -\frac{\chi_{i,r}^-\chi_{i,r}^+\chi_{j,s}^+}{(u+\tfrac{1}{2}+ w_{i,r})(w_{i,r}-w_{j,s}+\tfrac12)}+\frac{(w_{i,r}+w_{j,s}-\tfrac32)\chi_{i,r}^+\chi_{i,r}^-\chi_{j,s}^+}{(u-\tfrac{1}{2}+ w_{i,r})((w_{i,r}-\tfrac12)^2-w_{j,s}^2)}\\
&\hskip0.5cm +\frac{\chi_{i,r}^+\chi_{i,r}^-\chi_{j,s}^+}{(u-\tfrac{1}{2}+ w_{i,r})(w_{i,r}+w_{j,s}-\tfrac12)}-\frac{(w_{i,r}-w_{j,s}+\tfrac32)\chi_{i,r}^-\chi_{i,r}^+\chi_{j,s}^+}{(u+\tfrac{1}{2}+ w_{i,r})((w_{i,r}+\tfrac12)^2-w_{j,s}^2)}\bigg)\\
=&\sum_{r=1}^{\fkv_i}\bigg(\frac{2(w_{i,r}-1)\chi_{i,r}^+\chi_{i,r}^-\chi_{j,s}^+}{(u+\tfrac{1}{2}- w_{i,r})((w_{i,r}-\tfrac12)^2-w_{j,s}^2)} - \frac{2(w_{i,r}+1)\chi_{i,r}^-\chi_{i,r}^+\chi_{j,s}^+}{(u-\tfrac{1}{2}- w_{i,r})((w_{i,r}+\tfrac12)^2-w_{j,s}^2)}\\
&\hskip0.5cm-\frac{2(w_{i,r}+1)\chi_{i,r}^-\chi_{i,r}^+\chi_{j,s}^+}{(u+\tfrac{1}{2}+ w_{i,r})((w_{i,r}+\tfrac12)^2-w_{j,s}^2)} + \frac{2(w_{i,r}-1)\chi_{i,r}^+\chi_{i,r}^-\chi_{j,s}^+}{(u-\tfrac{1}{2}+ w_{i,r})((w_{i,r}-\tfrac12)^2-w_{j,s}^2)}\bigg).
\end{align*}
Recall $\W_{j,s}^\diamond(u)$ from \eqref{diamond} and set (cf. \eqref{eq:Omega})
\[
\Omega_{i;j,s}(u)=\frac{\bm\varkappa(u-\tfrac12)^{\vartheta_i}\W_{i-1}(u-\tfrac{1}{2})\W_{j,s}^\diamond(u-\tfrac{1}{2})\Z_i(u-\tfrac{1}2)}{\W_i(u)\W_i(u-1)}.
\]
Note that $\big((u-\hf)^2-w_{j,s}^2\big)\Omega_{i;j,s}(u)=\Omega_i(u)$.

We rewrite the above formula using \eqref{chi+-}--\eqref{chi-+} in two cases, i.e., $s=0$ and $s\gge 1$.
If $s=0$, then the above formula is equal to
\begin{align*}
\sum_{r=1}^{\fkv_i}\bigg(&\frac{(1-2w_{i,r})}{(w_{i,r}-\tfrac12)^2-w_{j,s}^2}\mathrm{Res}_{u=-w_{i,r}+\tfrac{1}{2}}\Omega_i(u+\hf)\Big(\frac{1}{u+\tfrac{1}{2}- w_{i,r}}+\frac{1}{u-\tfrac{1}{2}+ w_{i,r}}\Big)\\
&-\frac{(2w_{i,r}+1)}{(w_{i,r}+\tfrac12)^2-w_{j,s}^2}\mathrm{Res}_{u=w_{i,r}+\hf}\Omega_i(u+\hf)\Big(\frac{1}{u-\tfrac{1}{2}- w_{i,r}}+\frac{1}{u+\tfrac{1}{2}+ w_{i,r}}\Big)\bigg)\chi_{j,s}^+.
\end{align*}
If $s\gge 1$, then the above formula is equal to
\begin{align*}
\sum_{r=1}^{\fkv_i}\bigg(& \big(1-2w_{i,r}\big)\mathrm{Res}_{u=-w_{i,r}+\hf}\Omega_{i;j,s}(u+\hf)\Big(\frac{1}{u+\tfrac{1}{2}- w_{i,r}}+\frac{1}{u-\tfrac{1}{2}+ w_{i,r}}\Big)\\
&-\big(2w_{i,r}+1\big)\mathrm{Res}_{u=w_{i,r}+\hf}\Omega_{i;j,s}(u+\hf)\Big(\frac{1}{u-\tfrac{1}{2}- w_{i,r}}+\frac{1}{u+\tfrac{1}{2}+ w_{i,r}}\Big)\bigg)\chi_{j,s}^+.
\end{align*}
Now we consider the right-hand side of \eqref{serregl}. Then 
\begin{align*}
&\Phi_\mu^\la\big(B_j(-u+\tfrac12)H_i(u)-B_j(u+\tfrac12)H_i(u)\big)\\
=&\sum_{s=0}^{\fkv_j}\Big(\frac{\chi_{j,s}^+}{-u+1-w_{j,s}}+\frac{\chi_{j,s}^-}{-u+1+w_{j,s}}\\
&\hskip1cm-\frac{1}{u+1-w_{j,s}}\chi_{j,s}^+-\frac{1}{u+1+w_{j,s}}\chi_{j,s}^-\Big)\Omega_i(u+\hf)\\
=&-\frac{2u\Omega_i(u+\hf)}{u^2-w_{j,0}^2}\chi_{j,0}^+
\\&+\sum_{s=1}^{\fkv_j}\Omega_{i;j,s}(u+\hf)\bigg(\Big(\frac{u^2-(w_{j,s}-1)^2}{-u+1-w_{j,s}}-\frac{u^2-(w_{j,s}-1)^2}{u+1-w_{j,s}}\Big)\chi_{j,s}^+\\
&\hskip5.7cm+\Big(\frac{u^2-(w_{j,s}+1)^2}{-u+1+w_{j,s}}-\frac{u^2-(w_{j,s}+1)^2}{u+1+w_{j,s}}\Big)\chi_{j,s}^-\bigg)\\
=&-\frac{2u\Omega_i(u+\hf)}{u^2-w_{j,0}^2}\chi_{j,0}^+
-\sum_{s=1}^{\fkv_j}2u\Omega_{i;j,s}(u+\hf)(\chi_{j,s}^++\chi_{j,s}^-).
\end{align*}
Comparing the coefficients of $\chi_{j,0}^+$ and $\chi_{j,s}^+$ for $s\gge 1$, the Serre relation \eqref{serregl} follows from the formulas established above and Lemma \ref{lem:std} (see the end of the previous subsection where this lemma was applied similarly). Again, we need a case-by-case study. We only remark the following.
\begin{itemize}
    \item The term \eqref{inv0new} contributes only if $\theta_i=1$, in which case $\chi_{j,0}^+=0$ by the parity assumption \eqref{parity}.
    \item If $\chi_{j,0}^+\ne 0$, then $\W_j(u)$ is an odd function and by \eqref{vartheta} we have $\vartheta_i=1$. It is easy to see that $2u\Omega_i(u+\hf)$ is regular at $u=0$ and has simple zeros at $u=\pm w_{j,0}$. Similarly, $2u\Omega_{i;j,s}(u+\hf)$ is regular at $u=0$.
\end{itemize}

This completes the proof of Theorem \ref{thm:iGKLO:gl}.

\subsection{Completing the proof except quasi-split type A$_{2n}$}
\label{ssec:qspf}
It remains to prove Theorem \ref{thm:GKLOquasisplit}.

In this subsection, we exclude the quasi-split type A$_{2n}$. Note that $c_{i,\tau i}\ne -1$ for all $i\in \I$, the Serre relation \eqref{SerreIII2} does not show up and $\wp_i=0$ for all $i\in\I$. According to Lemma \ref{typeArel} and Remark \ref{rem:a3}, it is not hard to see that the verification of all the relations will be either similar to the case for the ordinary shifted Yangians, or the case for shifted iYangians of split type A. Here we only discuss for example the Serre relation for $c_{ij}=-1$ as the relation for $c_{ij}=0$ reduces to the case of shifted Yangian.

Suppose $c_{ij}=-1$. Note that $j\ne \tau i$, then we have several cases. First, one needs to verify that similar identities hold as in Lemma \ref{lem:chi}. We proceed case-by-case. 
\begin{enumerate}
    \item If both $i$ and $j$ are not fixed by $\tau$, then this is the case of ordinary shifted Yangian as verified in \cite[Appendix B]{BFN19}.
    \item If both $i$ and $j$ are fixed by $\tau$, then this is the split case as verified in \S\ref{sec:serre}.
    \item If $i$ is fixed by $\tau$ and $j$ is not, i.e., $i\in \I_0$ and $j\notin \I_0$, then we need to verify the relation \eqref{SerreIII}. Then the detail is parallel to that of \S\ref{sec:serre}. Again one needs to carefully deal with the terms containing $\chi_{i,0}^+\chi_{i,0}^+$ and $\chi_{i,r}^\pm\chi_{i,r}^\mp$ (those are scalar functions) while all other terms cancel due to Lemma \ref{lem:chi}.
    \item If $j$ is fixed by $\tau$ and $i$ is not, i.e., $i\notin \I_0$ and $j\in \I_0$, then we need to verify the relation \eqref{eq:Serre-ord}. Again, we argue as in \S\ref{sec:serre}. Since $i$ is not fixed by $\tau$ and $j\ne \tau i$, there is no way to obtain constant terms (scalar rational functions without difference operators). Thus this case essentially corresponds to the beginning of \S\ref{sec:serre}.
\end{enumerate}

\subsection{Completing the proof for the quasi-split type A$_{2n}$}
\label{ssec:qsAIII2n}

In this subsection, we complete the proof of Theorem \ref{thm:GKLOquasisplit} for the case of quasi-split type A$_{2n}$. We shall only verify the most complicated relations \eqref{bbNqs} and \eqref{SerreIII2} for the case $j=\tau i$ and $c_{i,\tau i}=-1$ for the iGKLO homomorphism in Theorem \ref{thm:GKLOquasisplit}. We prove the case $i\rightarrow j$ (as the other case $i\leftarrow j$ is similar). Then $\wp_i=-1$ and $\wp_j=1$. 

\subsubsection{The relation \eqref{bbNqs}}
We work with the corresponding relations in terms of generating series,
\be
\begin{split}
(u-v)[B_i(u),B_j(v)]=-\frac{1}{2}[B_i(u),B_j(v)]_+&+\big([B_i^{(1)},B_j(v)]-[B_i(u),B_j^{(1)}]\big)\\
&-\Big(\frac{2u}{u+v}\big(H_i(u)\big)^\star+\frac{2v}{u+v}\big(H_j(v)\big)^\star\Big),
\end{split}
\ee
see \cite[(3.21)]{LZ24}.  It can be equivalently written as
\beq\label{bbNqs-gen}
\begin{split}
(u-v+\tfrac12)B_i(u)B_j(v)&+B_i(u)B_j^{(1)}-B_i^{(1)}B_j(v)\\
= (u-v-\tfrac{1}{2})B_j(v)B_i(u)&+B_j^{(1)}B_i(u)-B_j(v)B_i^{(1)}\\
&-\Big(\frac{2u}{u+v}\big(H_i(u)\big)^\star+\frac{2v}{u+v}\big(H_j(v)\big)^\star\Big),
\end{split}
\eeq
cf. \eqref{eepf3}. We proceed as in \S\ref{ssec:bb} by introducing $\chi_{i,r}$ and $\chi_{j,s}$ for $1\lle r\lle \fkv_i$ and $1\lle s\lle \fkv_j$ such that
\beq\label{BIBJ}
\Phi_{\mu}^\la(B_i(u))=\sum_{r=1}^{\fkv_i}\frac{1}{u+\tfrac12-w_{i,r}}\chi_{i,r},\qquad \Phi_{\mu}^\la(B_j(u))=\sum_{s=1}^{\fkv_j}\frac{1}{u+\tfrac12-w_{j,s}}\chi_{j,s}.
\eeq
Note that we have used \eqref{signzw}. It is convenient to write $r':=\fkv_i+1-r$ for $1\lle r\lle \fkv_i$.
\begin{lem}\label{lem-chi-com}
If $r\ne s$, then we have
\begin{align*}
(w_{i,r}-w_{i,s}-1)\chi_{i,r}\chi_{i,s}&=(w_{i,r}-w_{i,s}+1)\chi_{i,s}\chi_{i,r},\\
(w_{i,r}-w_{j,s'}+\tfrac12)\chi_{i,r}\chi_{j,s'}&=(w_{i,r}-w_{j,s'}-\tfrac12)\chi_{j,s'}\chi_{i,r}.
\end{align*}
\end{lem}
\begin{proof}
Follows by a direct calculation.
\end{proof}

Thus it follows from the same calculation as in \S\ref{ssec:bb}  that all $\chi_{i,r}\chi_{j,s'}$ (or $\chi_{j,s'}\chi_{i,r}$) with $r\ne s$ in the LHS and RHS of  \eqref{bbNqs-gen} cancel. Hence we are left with terms involving $\chi_{i,r}\chi_{j,r'}$ and $\chi_{j,r'}\chi_{i,r}$ summed over $r$ and the scalar series from $H_i(u)$ and $H_j(v)$. Note that by \eqref{signzw}, we have $w_{i,r}=-w_{j,r'}$ and $\eth_{i,r}=\eth_{j,r'}^{-1}$ for $1\lle r\lle \fkv_i$. Thus these terms do not involve the difference operators.

The constant terms from the LHS of \eqref{bbNqs-gen} are given by
\begin{align*}
\frac{1}{u+v}\sum_{r=1}^{\fkv_i}\Big(\frac{2w_{i,r}-\frac12}{u+\frac12-w_{i,r}}-\frac{2w_{j,r'}+\frac12}{v-\frac12-w_{j,r'}}\Big)\chi_{i,r}\chi_{j,r'},
\end{align*}
while the constant terms from $(u-v-\tfrac{1}{2})B_j(v)B_i(u)+B_j^{(1)}B_i(u)-B_j(v)B_i^{(1)}$ is given by
\begin{align*}
\frac{1}{u+v}\sum_{r=1}^{\fkv_i}\Big(\frac{2w_{i,r}+\frac12}{u-\frac12-w_{i,r}}-\frac{2w_{j,r'}-\frac12}{v+\frac12-w_{j,r'}}\Big)\chi_{j,r'}\chi_{i,r}.
\end{align*}
Note that $i\rightarrow j$. Expressing $\chi_{i,r}\chi_{j,r'}$ and $\chi_{j,r'}\chi_{i,r}$ explicitly, we have 
\beq\label{rr'}
\begin{split}
\chi_{i,r}\chi_{j,r'}&=\frac{(2w_{i,r}-\tfrac32)\Z_i(w_{i,r}-\frac12)}{(2w_{i,r}-\frac12)\W_{i,r}(w_{i,r})\W_{i,r}(w_{i,r}-1)}\prod_{k\leftrightarrow i}\W_k(w_{i,r}-\tfrac12),\\
&=\frac{(2w_{j,r'}+\tfrac32)\Z_j(w_{j,r'}+\frac12)}{(2w_{j,r'}+\frac12)\W_{j,r'}(w_{j,r'})\W_{j,r'}(w_{j,r'}+1)}\prod_{k\leftrightarrow j}\W_k(w_{j,r'}+\tfrac12).
\end{split}
\eeq
The formulas for $\chi_{j,r'}\chi_{i,r}$ are similar but without the extra ratios (due to the choice of $i\rightarrow j$):
\beq\label{r'r}
\begin{split}
\chi_{j,r'}\chi_{i,r}&=\frac{\Z_j(w_{j,r'}-\frac12)}{\W_{j,r'}(w_{j,r'})\W_{j,r'}(w_{j,r'}-1)}\prod_{k\leftrightarrow j}\W_k(w_{j,r'}-\tfrac12),\\
&=\frac{\Z_i(w_{i,r}+\frac12)}{\W_{i,r}(w_{i,r})\W_{i,r}(w_{i,r}+1)}\prod_{k\leftrightarrow i}\W_k(w_{i,r}+\tfrac12).
\end{split}
\eeq
The rest follows the same type of argument at the end of \S\ref{ssec:eepf}, now with the help of Lemma \ref{lem:std}. We remark that the extra factor $1\pm\frac{1}{4u}$ accounts for the factors
\[
2u\big(1\pm \tfrac{1}{4u}\big)\Big|_{u=w_{i,r}\pm \frac12}=2w_{i,r}\pm \tfrac32,\qquad 2u\big(1\pm \tfrac{1}{4u}\big)\Big|_{u=w_{i,r}\mp \frac12}=2w_{i,r}\mp \tfrac12,
\]
and similar factors for $w_{j,r'}$.

\subsubsection{The relation \eqref{SerreIII2}}
\label{ssec:SerreIII}

By Lemma \ref{lem-reduction}, it suffices to verify the relation \eqref{SerreIII2} is preserved by $\Phi_\mu^\la$ for the case $s_1=s_2=s=1$. As first pointed out by \cite{SSX25}, it is actually easier to prove the following more general relation in generating function form
\beq\label{serreA2}
\big[B_i^{(1)},[B_i^{(1)},B_j(u)]\big]=\big(4u[B_i(3u),H_j(u)]\big)^\star,
\eeq
which implies that the relation \eqref{SerreIII2} is preserved by $\Phi_\mu^\la$ for $s_1=s_2=1$ and arbitrary $s>0$.

Recall the image of $B_i(u)$ from \eqref{BIBJ} and note that $w_{j,r'}=-w_{i,r}$. It follows from the identity
\begin{align*}
\frac{1}{3u+\frac12-w_{i,r}}\Big(\frac{u-w_{i,r}+1}{(u-\frac12+w_{i,r})(u-\frac32+w_{i,r})}-\frac{u-w_{i,r}}{(u+\frac12+w_{i,r})(u-\frac12+w_{i,r})}\Big)\\
=\frac{1}{(u+\frac12+w_{i,r})(u-\frac12+w_{i,r})(u-\frac32+w_{i,r})}
\end{align*}
that  
\begin{align}
    \label{Phi:LHS}
\Phi_\mu^\la \Big(4u[B_i(3u),H_j(u)] \Big) = \sum_{r=1}^{\fkv_i} \Xi_{i,r}(u) \chi_{i,r}, 
\end{align}
where 
\beq\label{Xidef}
\Xi_{i,r}(u):=\frac{(4u+1)\Z_j(u)\W_{i,r}(u)}{(u-\frac32+w_{i,r})\W_j(u+\frac12)\W_j(u-\frac12)}\prod_{\substack{k\leftrightarrow j\\k\ne i} }\W_{k}(u).
\eeq

On the other hand, the image of $\big[B_i^{(1)},[B_i^{(1)},B_j(u)]\big]$ under the map $\Phi_\mu^\la$ is
\beq\label{r1r2s}
\left[\sum_{r_1=1}^{\fkv_i}\chi_{i,r_1},\Big[\sum_{r_2=1}^{\fkv_i}\chi_{i,r_2},\sum_{s=1}^{\fkv_i}\frac{1}{u+\frac12-w_{j,s'}}\chi_{j,s'}\Big]\right].
\eeq
If $r_1\ne s$ and $r_2\ne s$, then a similar calculation as in \S\ref{sec:serre} using Lemma \ref{lem-chi-com} shows that
\[
\mathrm{Sym}_{r_1,r_2}\left[\chi_{i,r_1},\Big[\chi_{i,r_2},\frac{1}{u+\frac12-w_{j,s'}}\chi_{j,s'}\Big]\right]=0.
\]
Thus only the terms with either $r_1=s$ or $r_2=s$ in \eqref{r1r2s} survive and will make a nontrivial contribution to $\Phi_\mu^\la\big(\big[B_i^{(1)},[B_i^{(1)},B_j(u)]\big]\big)$.
Denote by $X_r$ the sum of the surviving terms containing $\chi_{i,r}$ from \eqref{r1r2s}, for $1\lle r\lle \fkv_i$. From the discussion above, we have 
\begin{align} \label{Phi:RHS}
\Phi_\mu^\la\Big(\big[B_i^{(1)},[B_i^{(1)},B_j(u)]\big]\Big) = \sum_{r=1}^{\fkv_i} X_r, 
\end{align}
where
\begin{align} \label{eq:XX}
X_r=X_r^\circ + \sum_{s=1,s\ne r}^{\fkv_i} (X'_s +X''_s), 
\end{align}
with
\begin{align*}
X_r^\circ&= \left[\chi_{i,r},\Big[\chi_{i,r},\frac{1}{u+\frac12-w_{j,r'}}\chi_{j,r'}\Big]\right], \qquad
X'_s =  \left[\chi_{i,r},\Big[\chi_{i,s},\frac{1}{u+\frac12-w_{j,s'}}\chi_{j,s'}\Big]\right],
\\
X''_s &=\left[\chi_{i,s},\Big[\chi_{i,r},\frac{1}{u+\frac12-w_{j,s'}}\chi_{j,s'}\Big]\right].
\end{align*}

The summand $X_r^\circ$ is equal to
\begin{align*}
X_r^\circ=\frac{1}{u-\frac32+w_{i,r}}\chi_{i,r}\chi_{i,r}\chi_{j,r'}-\frac{2}{u-\frac12+w_{i,r}}\chi_{i,r}\chi_{j,r'}\chi_{i,r}+\frac{1}{u+\frac12+w_{i,r}}\chi_{j,r'}\chi_{i,r}\chi_{i,r}.
\end{align*}
It follows from \eqref{rr'}--\eqref{r'r} (and recall $w_{j,r'}=-w_{i,r}$ and \eqref{Xidef} here) that it can be rewritten as
\beq\label{jojo1}
\begin{split}
X_r^\circ=\Big(\frac{1}{u-\frac32+w_{i,r}}\mathrm{Res}_{u=-w_{i,r}+\frac32}&+\frac{1}{u-\frac12+w_{i,r}}\mathrm{Res}_{u=w_{j,r'}+\frac12}\\&+\frac{1}{u+\frac12+w_{i,r}}\mathrm{Res}_{u=w_{j,r'}-\frac12}\Big)\,\Xi_{i,r}(u)\chi_{i,r}.
\end{split}
\eeq
Now let $1\lle s \lle \fkv_i$ with $s\neq r$. We have
\begin{align*}
X'_s =&\left[\chi_{i,r},\Big[\chi_{i,s},\frac{1}{u+\frac12-w_{j,s'}}\chi_{j,s'}\Big]\right]\\
=\,&\frac{\chi_{i,r}\chi_{i,s}\chi_{j,s'}-\chi_{i,s}\chi_{j,s'}\chi_{i,r}}{u-\frac12-w_{j,s'}}+\frac{\chi_{j,s'}\chi_{i,s}\chi_{i,r}-\chi_{i,r}\chi_{j,s'}\chi_{i,s}}{u+\frac12-w_{j,s'}}.
\end{align*}
On the other hand, it follows from Lemma \ref{lem-chi-com} that
\begin{align*}
X''_s =&\left[\chi_{i,s},\Big[\chi_{i,r},\frac{1}{u+\frac12-w_{j,s'}}\chi_{j,s'}\Big]\right]\\
=\ & \left[\chi_{i,s}, \frac{-1}{(u+\frac12-w_{j,s'})(w_{i,r}-w_{j,s'}+\frac12)}\chi_{j,s'}\chi_{i,r}\right]\\
=\ & \frac{-\chi_{i,s}\chi_{j,s'}\chi_{i,r}}{(u-\frac12-w_{j,s'})(w_{i,r}-w_{j,s'}-\frac12)}+\frac{\chi_{i,r}\chi_{j,s'}\chi_{i,s}}{(u+\frac12-w_{j,s'})(w_{i,r}-w_{j,s'}-\frac12)}
\end{align*}
Summing up the above two expressions, we have  
\begin{align*}
X'_s +X''_s &=\frac{1}{u-\frac12-w_{j,s'}}\Big(\chi_{i,r}\chi_{i,s}\chi_{j,s'}-\frac{w_{i,r}-w_{j,s'}+\frac12}{w_{i,r}-w_{j,s'}-\frac12}\,\chi_{i,s}\chi_{j,s'}\chi_{i,r}\Big)\\
&\quad+\frac{1}{u+\frac12-w_{j,s'}}\Big(\chi_{j,s'}\chi_{i,s}\chi_{i,r}-\frac{w_{i,r}-w_{j,s'}-\frac32}{w_{i,r}-w_{j,s'}-\frac12}\,\chi_{i,r}\chi_{j,s'}\chi_{i,s}\Big).
\end{align*}
Using \eqref{rr'}--\eqref{r'r} along with the identities
\begin{align*}
\frac{(w_{j,s'}+w_{i,r}+1)(w_{j,s'}+\frac32-w_{i,r})}{w_{j,s'}+w_{i,r}-1}&+(w_{i,r}-w_{j,s'}+\tfrac12)=\frac{4w_{j,s'}+1}{w_{j,s'}-1+w_{i,r}},\\
\frac{w_{j,s'}-\frac12-w_{i,r}}{(w_{j,s'}+w_{i,r})(w_{j,s'}-1+w_{i,r})}&+\frac{w_{i,r}-\frac32-w_{j,s'}}{(w_{j,s'}-1+w_{i,r})(w_{j,s'}-2+w_{i,r})}\\
&=\frac{4w_{j,s'}-1}{(w_{j,s'}+w_{i,r})(w_{j,s'}-1+w_{i,r})(w_{j,s'}-2+w_{i,r})},
\end{align*}
we can rewrite the previous equation as
\beq\label{jojo2}
X'_s +X''_s  =
\Big(\frac{1}{u-\frac12-w_{j,s'}}\mathrm{Res}_{u=w_{j,s'}+\frac12}+\frac{1}{u+\frac12-w_{j,s'}}\mathrm{Res}_{u=w_{j,s'}-\frac12}\Big)\,\Xi_{i,r}(u)\chi_{i,r}.
\eeq
Plugging \eqref{jojo1}--\eqref{jojo2} into \eqref{eq:XX} and applying Lemma \ref{lem:std} we obtain that $X_r =\big(\Xi_{i,r}(u)\big)^\star \chi_{i,r}$. Now comparing \eqref{Phi:LHS} with \eqref{Phi:RHS} completes the proof that the relation \eqref{SerreIII2} is preserved by $\Phi_\mu^\la$. 

This completes the proof of Theorem \ref{thm:GKLOquasisplit}.

\section{Identification of two definitions of TSTY}
\label{sec:2TSTY}

In this section, we make precise the connections between the TSTY for split type A (type AI) here with a distinguished family of TSTY defined in a very different way in \cite[Section~11]{LPTTW25}; that version of TSTY was motivated by its connection to finite W-algebras of classical type.

\subsection{Shifted twisted Yangians and TSTY from \cite{LPTTW25}}
\label{ssec:convLPTTW25}

Let $\g=\fksl_n$ and $\tau=\mathrm{id}$. We consider the special case $\lambda=N\varpi_1^\vee$ where $\varpi_1^\vee$ is the first fundamental coweight. Let $\mu$ be a dominant even coweight such that $N\varpi_1^\vee\gge \mu$. Recall $\bv_i\in\bN$ from \eqref{ell} and $\theta_i$ from \eqref{ell_theta}. Define $p_i\in\bN$ by 
\begin{align}  \label{eq:pi_vi}
p_1=\bv_{n-1},\quad p_{i}=\bv_{n-i}-\bv_{n-i+1},\quad p_n=N-\bv_1, \quad (1< i<n ).
\end{align}
Then $(p_n, p_{n-1}, \ldots, p_1)$ form a partition of $N$ such that $p_{i+1}-p_i=\langle \mu,\alpha_{n-i}\rangle$. It follows that all $p_i$ have the same parity.  The parity assumption defined in \cite[(11.1)]{LPTTW25} then is equivalent to the evenness condition of $\mu$. This partition determines  a pair $(\sigma,\ell)$, where $\ell=p_n$ and $\sigma=(\fks_{i,j})_{1\lle i,j\lle n}$ is a symmetric matrix satisfying $\fks_{i,i+1}=\hf(p_{i+1}-p_i)$ and $\fks_{i,j}+\fks_{j,m}=\fks_{i,m}$ provided $|i-j|+|j-m|=|i-m|$. We call $\sigma$ a shift matrix.

The shifted twisted Yangian of type AI, denoted by $\cy_n^+(\sigma)$, can be defined in parabolic presentation in \cite[\S8.1]{LPTTW25} with a composition of $n$ that is admissible to the shift matrix $\sigma$, or alternatively, in Drinfeld presentation in \cite[\S8.3]{LPTTW25}. It has been proved in \cite[Proposition~8.9]{LPTTW25} that the two shifted twisted Yangians are isomorphic. Here we only recall the Drinfeld presentation.

Let $C=(c_{ij})_{1\lle i,j<n}$ be the Cartan matrix of type A$_{n-1}$, and set $c_{0j}=-\delta_{j,1}$.
\begin{dfn}[{\cite[Theorem~5.3]{LWZ25}}]
The twisted Yangian of type AI is the algebra $\cy_n^+$ over $\bC$ 
generated by $\{\ch_{i}^{(r)}\}_{r> 0}$, $\{\cb_{j}^{(r)}\}_{r>0}$, for $0\lle i<n$ and $1\lle j<n$, subject to the following relations, for $r_1,r_2,r,s \in \bZ_{>0}$:
\begin{align}
[ \ch_{i}^{(r)}, \ch_{j}^{(s)}]&=0, \qquad { \ch_{i}^{(2r-1)}=0, } \label{drs1}\\	
[ \ch_{i}^{(r+1)}, \cb_{j}^{(s)}]-[ \ch_{i}^{(r-1)}, \cb_{j}^{(s+2)}]&=c_{ij}[ \ch_{i}^{(r)}, \cb_{j}^{(s+1)}]_+ +\frac{1}{4}c_{ij}^2[ \ch_{i}^{(r-1)}, \cb_{j}^{(s)}],\label{drs2}\\
[ \cb_{i}^{(r+1)}, \cb_{j}^{(s)}]-[ \cb_{i}^{(r)}, \cb_{j}^{(s+1)}]&=\frac{c_{ij}}{2}[\cb_{i}^{(r)}, \cb_{j}^{(s)}]_+-2\delta_{ij}(-1)^r \ch_{i}^{(r+s)},\label{drs3}\\
[\cb_{i}^{(r)},\cb_{j}^{(s)}]&=0,\qquad  \text{ for }|i-j|>1,\label{drs4}\\
\mathrm{Sym}_{r_1,r_2}\big[\cb_{i}^{(r_1)},[\cb_{i}^{(r_2)},\cb_{j}^{(r)}] \big]&= \notag
\\
(-1)^{r_1}\sum_{p\gge 0}2^{-2p}  \big([\ch_{i}^{(r_1+r_2-2p-1)}&,\cb_{j}^{(r+1)}]-[\ch_{i}^{(r_1+r_2-2p-1)},\cb_{j}^{(r)}]_+\big),
\quad \text{if } c_{i,j}=-1,
 \label{drs5}
\end{align}
for all admissible indices $i,j,r,s$. Here by convention, $\ch_{i}^{(0)}=1$.
\end{dfn}

\begin{dfn}[{\cite[Definition~8.7]{LPTTW25}}]
The (\textit{dominantly}) \textit{shifted Drinfeld twisted Yangian} associated to the shift matrix $\sigma$ is the algebra $\cy_n^+(\sigma)$ over $\bC$ 
generated by $\{\ch_{i}^{(r)}\}_{r> 0}$, $\{\cb_{j}^{(r)}\}_{r>\fks_{j+1,j}}$, for $0\lle i<n$ and $1\lle j<n$, subject to the relations \eqref{drs1}--\eqref{drs5}, for all admissible indices $r_1,r_2,r,s \in \bZ_{>0}$.
\end{dfn}

The shifted twisted Yangian $\cy_n^+(\sigma)$ can be naturally identified as a subalgebra of $\cy_n^+$ by identifying elements with the same symbols.

The twisted Yangian $\cy_n^+$ also contains mutually commuting elements $\mathscr D_i^{(r)}$ for $1\lle i\lle n$ and $r>0$ defined as follows
\[
\ch_0(u)=\mathscr D_1(u),\qquad \ch_i(u)=\big(\mathscr D_{i}(u-\tfrac{i}2)\big)^{-1}\mathscr D_{i+1}(u-\tfrac{i}{2}),\qquad 1\lle i<n,
\]
where
\[
\ch_i(u):=1+\sum_{r>0}\ch_i^{(r)}u^{-r},\qquad 
\mathscr D_j(u):=1+\sum_{r>0}\mathscr D_j^{(r)}u^{-r},\quad 0\lle i<n,~1\lle j\lle n,
\]
see \cite[\S3.2]{LPTTW25} or \cite[\S3]{LWZ25}.

Introduce $\mathscr Q_i(u)$ for $1\lle i\lle n$ by the rule
\[
\mathscr Q_i(u):=\mathscr D_1(u+\tfrac{i-1}{2})\mathscr D_2(u+\tfrac{i-3}{2})\cdots \mathscr D_{i-1}(u-\tfrac{i-3}{2})\mathscr D_i(u-\tfrac{i-1}{2}).
\]
Then
\beq\label{H-in-Q}
\ch_0(u)=\mathscr D_1(u)=\mathscr Q_1(u),\qquad \ch_i(u)=\frac{\mathscr Q_{i-1}(u)\mathscr Q_{i+1}(u)}{\mathscr Q_i(u-\hf)\mathscr Q_i(u+\hf)},\qquad 1\lle i<n.
\eeq
Moreover, the series $\mathscr Q_n(u)$ corresponds to the Sklyanin determinant $\mathrm{sdet}\,S(u+\tfrac{n-1}{2})$ which is an even series, see \cite[Theorems 2.5.3, 2.12.1]{Mol07}. 

Let $\mathscr{SY}_n^+$ be the subalgebra of $\cy_n^+$ generated by  $\{\ch_{i}^{(r)},\cb_{i}^{(r)}\}_{r>0}$, for $0< i<n$. Denote by $\mathscr{ZY}_n^+$ the center of $\cy_n^+$. Then it is known \cite[Theorem~5.1]{LWZ25} that $\mathscr{SY}_n^+$ is isomorphic to the algebra generated by $\{\ch_{i}^{(r)},\cb_{i}^{(r)}\}_{r>0}$ for $0< i<n$ subject to the relations \eqref{drs1}--\eqref{drs5}.
Define $\mathscr Z^{(r)}$ for $r>0$ by
\beq\label{Q-expand}
\mathscr Q_n(u)=1+\sum_{r>0}\mathscr Z^{(r)}u^{-2r}.
\eeq
Then it is known \cite[Theorem~2.8.2]{Mol07} that the center of $\cy_n^+$ is freely generated by the elements $\mathscr Z^{(r)}$ for $r>0$. Thus it follows from \cite[Theorem~2.9.2]{Mol07} that
\beq
\cy_n^+=\mathscr{SY}_n^+\otimes\mathscr{ZY}_n^+=\mathscr{SY}_n^+\otimes\bC[\mathscr Z^{(r)}\mid r>0].
\eeq

Similarly, let $\mathscr{SY}_n^+(\sigma)$ be the subalgebra of $\cy_n^+$ generated by  $\{\ch_{i}^{(r)}\}_{r> 0}$, $\{\cb_{i}^{(r)}\}_{r>\fks_{i,i+1}}$, for $0< i<n$. Denote by $\mathscr{ZY}_n^+(\sigma)$ the center of $\cy_n^+(\sigma)$. By passing to the associated graded (the loop filtration), one obtains that $\mathscr{ZY}_n^+(\sigma)=\mathscr{ZY}_n^+$ (see e.g. \cite[Remark~8.14]{LPTTW25})  and
\beq\label{sigma-dec}
\cy_n^+(\sigma)=\mathscr{SY}_n^+(\sigma)\otimes\mathscr{ZY}_n^+=\mathscr{SY}_n^+(\sigma)\otimes\bC[\mathscr Z^{(r)}\mid r>0].
\eeq

\begin{lem}\label{lem:kappa}
The map $\kappa:\Yt_\mu\to \mathscr{SY}_n^+(\sigma)$ defined by the rule
\[
 H_{n-i}^{(r)}\mapsto \ch_i^{(r+2\fks_{i,i+1})},\quad B_{n-i}^{(s)}\mapsto \sqrt{(-1)^{\fks_{i,i+1}}}\cb_i^{(s+\fks_{i,i+1})},
\]
for $0<i<n$, $r>-2\fks_{i,i+1}$, $s>0$, uniquely induces an isomorphism of algebras. 
\end{lem}
\begin{proof}
It is clear that the map induces an epimorphism. Then it follows from comparing the PBW bases in Theorem \ref{thm:pbw_arb} and \cite[Proposition~3.14]{LWZ25} that it is indeed an isomorphism. 
\end{proof}


Then we recall the truncation from \cite[\S11.2]{LPTTW25} in terms of the Drinfeld presentation rather than the parabolic presentation, thanks to \cite[Proposition~11.5]{LPTTW25}. In this case, the composition $\nu$ is simply $(1,\dots,1)$, the element  $H_{1;1,1}^{(r)}$ therein corresponds to $\ch_{0}^{(r)}$ for $r>0$ and $B_{1;1,1}^{(s)}$ corresponds to $\cb_{1}^{(s)}$ for $s>\fks_{1,2}$. 
Introduce elements $\wtl\cb_{i}^{(r+\fks_{1,2})}$ for $r>0$ by the rule
\beq\label{b-def2}
\cb_1(u)=\sum_{r> 0}\cb_{1}^{(r+\fks_{1,2})}u^{-r},
\qquad 
\wtl \cb_1(u):= \cb_1(u+\tfrac12)\ch_0(u)= \sum_{r>0}\wtl\cb_{i}^{(r+\fks_{1,2})}u^{-r}.
\eeq
Thus the element $\wtl B_{1;1,1}^{(\fks_{1,2};p_1+1)}$ corresponds to $\wtl\cb_{1}^{(\fks_{1,2}+p_1+1)}$. 
Then the truncated shifted twisted Yangian (TSTY) $\cy_{n,\ell}^+(\sigma)$ in \cite{LPTTW25} is defined to be the quotient
\beq\label{W-TSTY}
\cy_{n,\ell}^+(\sigma):=\cy_{n}^+(\sigma)/\mathscr I_{\ell},
\eeq
where $\mathscr I_\ell$ is the 2-sided ideal of $\cy_{n}^+(\sigma)$ generated by
\beq\label{W-ideal}
\{\ch_{0}^{(r)}\mid r>p_1\}\cup \{\delta_{\bar p_1,\bar 0}\wtl \cb_{1}^{(\fks_{1,2}+p_1+1)}\}.
\eeq
It has been proved that in \cite[Theorem F]{LPTTW25} that $\cy_{n}^+(\sigma)$ are isomorphic to finite W-algebra of classical types. Let $k=\lfloor\tfrac{N}{2}\rfloor$. The precise type is determined as follows:
\begin{align}  \label{BCDk}
    \begin{cases}
        \mathrm{C}_k, & \text{ if all $p_i$ are even},
        \\
        \mathrm{B}_k, & \text{ if all $p_i$ and $n$ are odd},
        \\
        \mathrm{D}_k, & \text{ if all $p_i$ are odd and $n$ are even}.
    \end{cases}
\end{align}

We shall slightly modify the iGKLO representations and establish the connections between the TSTY $\Yt_\mu^{N\varpi_1^\vee}$ and $\cy_n^+(\sigma)$.

Under the choice of the data, we have $\lambda=N\varpi_1^\vee$ and hence $\bw_1=N$ and $\bw_i=0$ for $i>1$. Thus by the conventions from \S3.1--\S3.2, only for $i=1$, the parameters $\fkw_i,\varsigma_i$ and the polynomials $\Z_i(u)$ are nontrivial. For that reason we shall drop the index $i=1$. Let $\varsigma=N-2k$. We have
\beq\label{Zdfn}
\Z(u):=\Z_1(u)=u^\varsigma\prod_{s=1}^{k}(u-z_s^2),\qquad \Z_i(u)=1,
\eeq
for $i>1$. 

Recall that all $p_i$ have the same parity, which we denote by $\theta$. It is easy to see from \eqref{ell_theta}--\eqref{vartheta} that $\vartheta_i=\theta$ for all $i\in \I$. Hence the action of $H_i(u)$ in the iGKLO representation $\Phi_\mu^{N\varpi_1^\vee}$ is given as follows
\[
H_i(u)\mapsto \frac{\bm\varkappa(u)^{\theta}\Z_i(u)}{\W_{i}(u-\tfrac{1}2)\W_{i}(u+\tfrac{1}2)}\prod_{j\leftrightarrow i}\W_j(u),\qquad 1\lle i<n.
\]

We introduce some variant of the algebra $\Yt_\mu[\bm z]$:
\begin{equation*}
\Yt_\mu[\bm z^2]^{S_k}:=\Yt_\mu\otimes \bC[z_1^2,\dots,z_k^2]^{S_k},\qquad 
\wY_\mu^{N\varpi_1^\vee}:= \Phi_\mu^{N\varpi_1^\vee}\big(\Yt_\mu[\bm z^2]^{S_k}\big),
\end{equation*}
where $\bC[z_1^2,\dots,z_k^2]^{S_k}$ is the algebra of symmetric polynomials in $z_1^2,\dots,z_k^2$.

Introduce formal variables $\mathsf Z^{(r)}$ for $r>0$ and consider the tensor product of algebras
\beq\label{extension1}
\Yt_\mu[\bm\sfZ]:=\Yt_\mu\otimes \bC[\mathsf Z^{(r)}\mid r>0].
\eeq
Set
\beq\label{extension2}
\bm\sfZ(u)=1+\sum_{r>0}\sfZ^{(r)}u^{-2r}.
\eeq
Now we have obtained four algebras $\cy_n^+(\sigma)$, $\Yt_\mu[\bm\sfZ]$, $\Yt_\mu[\bm z^2]^{S_k}$, $\wY_\mu^{N\varpi_1}$. Since both $\Yt_\mu[\bm\sfZ]$ and $\Yt_\mu[\bm z^2]^{S_k}$ are extensions of $\Yt_\mu$ by tensoring with a polynomial algebra, it is easy to see we have a surjective homomorphism $\Yt_\mu[\bm\sfZ]\twoheadrightarrow \Yt_\mu[\bm z^2]^{S_k}$ via specializing $\bm\sfZ(u)$ to $u^{-N}\Z(u)$, where $\Z(u)$ is given in \eqref{Zdfn}. 

Introduce elements $\mathsf{A}_i^{(r)}$ and $\mathsf{B}_i^{(r)}$ in $\Yt_\mu[\bm z^2]^{S_k}$ for $r>0$, respectively, by \eqref{GKLO-A} and \eqref{GKLO-B}. Also, introduce $\mathsf{A}_i^{(r)}$ (we use the same notation as they are canonically identified under the quotient map $\Yt_\mu[\bm\sfZ]\twoheadrightarrow \Yt_\mu[\bm z^2]^{S_k}$) in $\Yt_\mu[\bm\sfZ]$ for $r>0$ by
\beq\label{GKLO-A2}
H_i(u)=\frac{\bm\varkappa(u)^{\theta}(u^N\bm\sfZ(u))^{\delta_{1,i}}\prod_{j\leftrightarrow i}u^{\bv_j}}{(u^2-\frac14)^{\bv_i}}\frac{\prod_{j\leftrightarrow i} \mathsf A_j(u)}{\mathsf A_i(u-\frac12)\mathsf A_i(u+\frac12)},
\eeq
where $\mathsf A_i^{(r)}$ are the coefficients of $\mathsf A_i(u)$ as in \eqref{A-coeff}. Recall the numbers $\bv_i$ from \eqref{ell}. Define a family of new numbers $q_i$ for $0\lle i<n$ by
\beq\label{q-def}
q_i:=(\bv_i-\theta)+2\sum_{j=1}^{n-i-1}(-1)^j(\bv_{i+j}-\theta),
\eeq
where by convention $\bv_0=N$. If $\theta=0$, then all $\theta_i=0$. If $\theta=1$, then $\theta_{i}$ is equal to the parity of $n-i$.
\begin{lem}\label{lem:iso-ext}
There exists a unique isomorphism  $\tilde\kappa: \Yt_\mu[\bm\sfZ] \rightarrow\cy_n^+(\sigma)$ which extends $\kappa$ from Lemma \ref{lem:kappa} by letting
\[
\bm\sfZ(u)\mapsto \mathscr Q_n(u)\prod_{i=1}^{n-1}\big(1-(\tfrac{i}{2u})^2\big)^{q_i}.
\]
Moreover, under the isomorphism $\tilde\kappa$, we have
\[
\mathsf A_{n-i}(u)\mapsto \mathscr Q_i(u)\prod_{j=1}^{i-1}\big(1-(\tfrac{i-j}{2u})^2\big)^{q_{n-j}},\qquad 1\lle i<n.
\]
\end{lem}
\begin{proof}
The first statement follows from Lemma \ref{lem:kappa}, \eqref{Q-expand}, \eqref{sigma-dec}, and \eqref{extension1}--\eqref{extension2}.

Now we prove the second statement. Denote
\beq\label{fkQ}
\mathfrak P_i(u):=\prod_{j=1}^{i-1}\big(1-(\tfrac{i-j}{2u})^2\big)^{q_{n-j}},\qquad 1\lle i\lle n.
\eeq
Recall the definition of $\mathsf A_i(u)$ from \eqref{GKLO-A2} and the relation 
\[
\kappa(H_{i}(u))=u^{\langle\mu,\alpha_i\rangle}\ch_{n-i}(u)=u^{\langle\mu,\alpha_i\rangle}\frac{\mathscr Q_{n-i-1}(u) \mathscr Q_{n-i+1}(u)}{\mathscr Q_{n-i}(u-\hf)\mathscr Q_{n-i}(u+\hf)}
\]
from \eqref{H-in-Q}. Thus it suffices to verify that
\[
\big(1-(\tfrac{1}{2u})^2\big)^{\theta-\bv_i}\frac{\mathfrak P_{n-i-1}(u)\mathfrak P_{n-i+1}(u)}{\mathfrak P_{n-i}(u-\hf)\mathfrak P_{n-i}(u+\hf)}=1,
\]
which follows from a straightforward calculation using \eqref{q-def} and \eqref{fkQ}.
\end{proof}

\begin{rem}
If all $p_i$ are even, then $\theta_i=0$ and the polynomial
\beq\label{P-W}
\mathscr P_n(u):=u^{q_0} \prod_{j=1}^{n-1} \big(u^2-(\tfrac{i}{2})^2\big)^{q_i}
\eeq
coincides with the polynomial in \cite[Proposition~12.5]{LPTTW25} for $\ell$ even (associated to the symmetric pyramid determined by the partition $p_1\lle p_2\lle\cdots\lle p_n$). 

Similarly, if all $p_i$ are odd, then the polynomial $(u+\hf)\mathscr P_n(u)$ coincides with the polynomial in \cite[Proposition~12.5]{LPTTW25} for $\ell$ odd. Note that it is not hard to prove that the extra factor $u+\rho_0=u+\hf$ there (for $\ell$ odd) can be removed so that the RHS of \cite[Proposition~12.5]{LPTTW25} remains to be a polynomial in $u$.
\end{rem}

\subsection{Identifying two TSTY's}

We introduced a variant of the TSTY $\Yt_\mu^{N\varpi_1^\vee}$, by restricting the allowed elements coming from $\bC[\bm z]$:
\begin{equation}
\label{def:tildeTSTY}
\wY_\mu^{N\varpi_1^\vee}:= \Phi_\mu^{N\varpi_1^\vee}\big(\Yt_\mu[\bm z^2]^{S_k}\big).
\end{equation}

\begin{prop}\label{prop:surjective}
There exists an epimorphism $\cy_{n,\ell}^+(\sigma) \twoheadrightarrow \wY_{\mu}^{N\varpi_1^\vee}$.
\end{prop}
\begin{proof}
By the discussion above, the following composition,
\[
\cy_n^+(\sigma) \stackrel{\sim}{\rightarrow} \Yt_\mu[\bm\sfZ]\twoheadrightarrow \Yt_\mu[\bm z^2]^{S_k} \twoheadrightarrow \wY_{\mu}^{N\varpi_1^\vee},
\]
is an epimorphism, where the first isomorphism is from Lemma \ref{lem:iso-ext}, the second epimorphism is obtained by specializing $\bm{\mathsf Z}(u)$ to $u^{-N}\Z(u)$, and the last one by definition. 

Note that the two-sided ideal generated by $\mathsf A_i^{(r)}$, $\delta_{\bar{\bv}_i, \bar 0}\mathsf B_{i}^{(r)}$ for  $i \in \I$, $r > \bv_i$ is contained in the kernel of the quotient $\Yt_\mu[\bm z^2]^{S_k} \twoheadrightarrow \wY_{\mu}^{N\varpi_1^\vee}$. By Lemma \ref{lem:iso-ext}, we find that $\ch_0(u)=\mathscr Q_1(u)$ in $\cy_n^+(\sigma)$ is sent to $\mathsf A_{n-1}(u)$ in $\Yt_\mu[\bm z^2]^{S_k}$. If $\bv_{n-1}$ is even, then it follows from Lemma \ref{lem:kappa}, \eqref{GKLO-B}, and \eqref{b-def2} that $\wtl{\mathscr{B}}_{1}^{(\fks_{1,2}+p_1+1)}$ in $\cy_n^+(\sigma)$  is sent to
$\mathsf B_{n-1}^{(p_1+1)}=\mathsf B_{n-1}^{(\bv_{n-1}+1)}$ in $\Yt_\mu[\bm z^2]^{S_k}$. Thus it follows from \eqref{W-TSTY}-\eqref{W-ideal} that the ideal $\mathscr I_\ell$ is sent to zero. Thus we have an epimorphism $\cy_{n,\ell}^+(\sigma)\twoheadrightarrow \wY_{\mu}^{N\varpi_1^\vee}$.
\end{proof}

Recall from \cite[Corollary 11.11]{LPTTW25}, that $\cy_{n,\ell}^+(\sigma)$ quantizes the Slodowy slice $\mc S^\epsilon_\pi$ for classical Lie algebra $\fksl_N^\epsilon$ in \cite[\S 6.1]{LWW25islice}, where $\pi$ is the partition $(p_n,\ldots, p_1)$ of $N$; cf. \eqref{eq:pi_vi}.  

\begin{lem}  \label{lem:dim}
    The number of PBW generators for $\cy_{n,\ell}^+(\sigma)$ is equal to $\dim \mc S^\epsilon_\pi$, which is given by
    $ 2\sum_{i=1}^{n-1} \fkv_i +k.$
\end{lem}

\begin{proof}
The first equality holds by \cite[Corollary 11.11]{LPTTW25}. 
By \cite[Corollary 11.9]{LPTTW25} (and taking the admissible shape there to be $(1,1,\ldots,1)$), this number of PBW generators is given by
    \begin{align} \label{eq:dim}
        \sum_{a=1}^{n-1} (n-a)p_a +\sum_{a=1}^n \lfloor \tfrac{p_a}2 \rfloor.
    \end{align}

    By \eqref{eq:pi_vi}, we have \begin{align} \label{eq:bvi:pi}
        \sum_{a=1}^b p_a =\bv_{n-b}, \quad\text{ for }1\lle b\lle {n-1}, 
    \end{align}
     and thus by \eqref{ell_theta} we have
     \begin{align}
         \label{eq:partialsum}
     \sum_{a=1}^{n-1} (n-a)p_a =\sum_{i=1}^{n-1} \bv_i =2\sum_{i=1}^{n-1} \fkv_i +\sum_{i=1}^{n-1} \theta_i. 
     \end{align}

    Recall that all $p_a$ have the same parity, $\sum_{a=1}^n p_a=N$, and $k=\lfloor \tfrac{N}2 \rfloor$. Calculating $\theta_i$ in \eqref{ell_theta} using \eqref{eq:bvi:pi} case-by-case, we have 
\begin{align*}
    \sum_{i=1}^{n-1} \theta_i +\sum_{a=1}^n \lfloor \tfrac{p_a}2 \rfloor
    =\begin{cases}
        0 +N/2 =k, & \text{ if all $p_i$ are even},
        \\
        (n-1)/2+(N-n)/2=k, & \text{ if all $p_i$ and $n$ are odd},
        \\
        n/2 +(N-n)/2=k, & \text{ if all $p_i$ are odd and $n$ are even}.
    \end{cases} 
\end{align*}
The lemma now follows by plugging this last formula and  \eqref{eq:partialsum} into \eqref{eq:dim}. 
\end{proof}

\begin{thm}
    \label{thm:TruncatedSTY}
    The two TSTY's $\cy_{n,\ell}^+(\sigma)$ and $ \wY_{\mu}^{N\varpi_1^\vee}$ are isomorphic.
\end{thm}

\begin{proof}
   In the proof, we will heavily use the notion of \emph{Gelfand-Kirillov dimension} $\operatorname{GKdim} A$ of an algebra $A$.   We refer the reader to \cite{KL00} for a detailed overview. 

   We shall prove that the epimorphism $\cy_{n,\ell}^+(\sigma) \twoheadrightarrow \wY_{\mu}^{N\varpi_1^\vee}$ from Proposition \ref{prop:surjective} is an isomorphism.

    By \cite[Corollary 11.11]{LPTTW25}, there is an isomorphism $\gr' ~\cy_{n,\ell}^+(\sigma) \cong \bC[\mc S^\epsilon_\pi]$ where $\mc S^\epsilon_\pi$ is a Slodowy slice.  Since $\bC[\mc S^\epsilon_\pi]$ is a polynomial ring, it follows that $\cy_{n,\ell}^+(\sigma)$ is  a domain. The algebra $\cy_{n,\ell}^+(\sigma)$ has a PBW basis in \cite[Corollary 11.9]{LPTTW25}, and thus its Gelfand-Kirillov dimension can be computed by counting the number of PBW generators, which by Lemma \ref{lem:dim} is equal to 
    $$
    \operatorname{GKdim} \cy_{n,\ell}^+(\sigma) = \dim \mc S^\epsilon_\pi= 2 \sum_{i=1}^{n-1} \fkv_i + k.
    $$    
{\bf Claim.} We have $\operatorname{GKdim} \wY_\mu^{N\varpi_1^\vee} \gge 2 \sum_{i=1}^{n-1} \fkv_i + k$.  
    
    Assuming this claim for the moment,  we deduce that
    \begin{equation}
    \label{eq:GKdimieq}
    \operatorname{GKdim} \cy_{n,\ell}^+(\sigma)  \lle \operatorname{GKdim} \wY_\mu^{N\varpi_1^\vee}.
    \end{equation}
    Since $\cy_{n,\ell}^+(\sigma)$ is a domain, by \cite[Proposition 3.15]{KL00} the Gelfand-Kirillov dimension of any of its proper quotients is strictly smaller. From the inequality \eqref{eq:GKdimieq}, we conclude that the epimorphism $\cy_{n,\ell}^+(\sigma) \twoheadrightarrow \wY_\mu^{N\varpi_1^\vee}$ must be an isomorphism.

    It remains to prove the claim. Consider the quotient filtration $F_{\mu_1}^\bullet \wY_\mu^{N\varpi_1^\vee}$ inherited from $\Yt_\mu[\bm z^2]^{S_k} $ as in \cite[\S2.5]{LWW25islice}.   Since Gelfand-Kirillov dimensions can only decrease upon passing to associated graded algebras \cite[Lemma 6.5]{KL00}, we have $$\operatorname{GKdim} \wY_\mu^{N \varpi_1^\vee} \gge \operatorname{GKdim} \gr ~\wY_\mu^{N \varpi_1^\vee}.$$     It is also a general fact that if $f : A \rightarrow B$ is a filtered map of filtered algebras with associated graded $\gr f : \gr A \rightarrow \gr B$, then $\operatorname{Im}(f)$ inherits a quotient filtration from $A$, and there is a surjection of graded algebras $\gr \operatorname{Im}(f) \twoheadrightarrow \operatorname{Im}( \gr f)$.   Applied to the map $ \Phi_\mu^{N\varpi_1^\vee} : \Yt_\mu[\bm z^2]^{S_k} \rightarrow \mc A$, we obtain 
    $$
    \gr ~\wY_\mu^{N \varpi_1^\vee} = \gr \operatorname{Im}( \Phi_\mu^{N \varpi_1^\vee}) \twoheadrightarrow \operatorname{Im}( \gr~ \Phi_\mu^{N \varpi_1^\vee}).
    $$
    It follows from \cite[Theorem~5.12]{LWW25islice} that the Gelfand-Kirillov dimension of $\operatorname{Im}( \gr~ \Phi_\mu^\la)$ is exactly $2 \sum_i \fkv_i + k$. (Recall that for a commutative algebra the Gelfand-Kirillov dimension is equal to the Krull dimension \cite[Proposition 7.9]{KL00}.)  Altogether we have
    $$
    \operatorname{GKdim} \wY_\mu^{N \varpi_1^\vee}\gge \operatorname{GKdim} \gr ~\wY_\mu^{N \varpi_1^\vee} \gge \operatorname{GKdim} \operatorname{Im}(\gr~\Phi_\mu^{N \varpi_1^\vee}) = 2 \sum_{i=1}^{n-1} \fkv_i + k,
    $$
    proving the claim. This completes the proof of the theorem. 
\end{proof}

\bibliographystyle{amsalpha}
\bibliography{references}
\end{document}